\documentclass{birkjour}
\usepackage{amssymb, amsmath}
 \newtheorem{thm}{Theorem}[section]
 \newtheorem{cor}[thm]{Corollary}

 \newtheorem{lemma}[thm]{Lemma}
 
 \theoremstyle{definition}
 \newtheorem{defn}[thm]{Definition}
 \theoremstyle{rem}{
 \newtheorem*{rem}{Remark}
 \newtheorem*{ex}{Example}
 \numberwithin{equation}{section}

\begin{document}

%
%
%
%
%
%
%
%
%

\title[Differentiation of IDMC Measures]{A Theory of Intermittency Differentiation of 1D Infinitely Divisible Multiplicative Chaos Measures}

\author[D. Ostrovsky]{Dmitry Ostrovsky}


\email{dm\_ostrov@aya.yale.edu}




\date{December 29, 2017}

\begin{abstract}
A theory of intermittency differentiation is developed for a general class of 1D Infinitely Divisible Multiplicative Chaos 
measures. The intermittency invariance of the underlying infinitely divisible field is established and utilized to
derive a Feynman-Kac equation for the distribution of the total mass of the limit 
measure by considering a stochastic flow in intermittency. The 
resulting equation prescribes the rule of intermittency differentiation for a general functional of the total mass
and determines the distribution of the total mass and its dependence structure to the first order in intermittency. 
A class of non-local functionals of the limit measure extending the total mass is introduced and shown to be invariant 
under intermittency differentiation making the computation of the full high temperature expansion of the total mass distribution possible in principle. 
For application,
positive integer moments and covariance structure of the total mass are considered in detail.    
\end{abstract}

\maketitle


\section{Introduction}
\noindent
In this paper we advance the theory of Infinitely Divisible Multiplicative Chaos (IDMC) measures, also known as 
limit log-infinitely divisible random measures, on the unit interval. Their study was initiated by 
Mandelbrot \cite{secondface}, \cite{Lan} and Bacry \emph{et. al.} \cite{MRW} in the limit lognormal case, extended to the compound Poisson case 
by Barral and Mandelbrot \cite{Pulses}, and developed in the general infinitely divisible case by Bacry and Muzy \cite{BM1}, \cite{BM} based 
on the theory of Kahane
\cite{K2}, spectral representations of infinitely divisible processes of Rajput and Rosinski \cite{RajRos}, and
conical set constructions of Barral and Mandelbrot \cite{Pulses} and Schmitt and Marsan \cite{Schmitt}, and recently
extended to multiple dimensions by Chainais \cite{Chainais} and Rhodes and Vargas \cite{RVmulti}.
A different conical construction was introduced by Barral and Jin \cite{barralID}, who studied the problems of
non-degeneracy of the limit measure and of the finiteness of its positive and negative integer moments in their model of 
infinitely divisible multiplicative chaos. We noted in \cite{Me5} that the underlying infinitely divisible field of the Bacry-Muzy construction
exhibits certain invariances and derived explicit multiple integral representations for single and joint positive integer moments of the total mass of the Bacry-Muzy measure in \cite{MeLMP}. 

\subsection{Why multiplicative chaos?}
The interest in multiplicative chaos measures stems from their properties of stochastic self-similarity (also known as continuous dilation invariance), of being grid-free and stationary, and having exactly multiscaling moments. What this means is that 
the random measure $M(t, t+s)$ of the interval $(t, t+s)$
satisfies the equality in law
\begin{equation}\label{cascade}
M(t, t+\gamma s) = W_\gamma M(t, t+s)
\end{equation}
and its distribution is independent of $t.$ The multiplier $W_\gamma$ is a positive stochastic multiplier that is stochastically  
independent of $M(t, t+s),$ $0<s<1$ and $0<\gamma\leq 1.$ The  multiplier
must necessarily be log-infinitely divisible  as was first pointed out by Novikov \cite{Novikov} long before such processes were
constructed mathematically.  The physical meaning of this equation is that the law of the process on smaller scales
is determined by that on the larger scales, \emph{i.e.} it is a continuous multifractal cascade. Moreover, it is clear that the scaling 
exponent of the moments of $M(t, t+\gamma s)$ as a function of the scale $\gamma$ is determined by the law of the multiplier
and in general is not a linear function of the moment order $q,$ \emph{i.e.} the process exhibits multiscaling.  
\begin{equation}
\textbf{E} [M(t, t+\gamma s)^q] = const(q, s) \,\textbf{E}[W_\gamma^q]. 
\end{equation} 
For example, in the case of Gaussian multiplicative chaos, $W_\gamma$ is lognormal so that the multiscaling spectrum is
a parabola \cite{MRW}, \cite{secondface}, \cite{Lan}.

Due to these remarkable properties
and the complexity of mathematical problems that they pose such as understanding the stochastic dependence structure
of the limit measure, multiplicative chaos measures are generating a significant level of interest in mathematical physics, especially in the context of conformal field theory and Liouville quantum gravity, cf. \cite{Jones}, \cite{BenSch}, \cite{DS}, \cite{RV1}, \cite{RV3} and in statistical modeling of fully intermittent 
turbulence \cite{Frisch}. In fact, the latter
application was the primary motivating factor for the introduction of multifractal processes and its continued development. 
For example, Mandelbrot's celebrated limit lognormal model of energy dissipation \cite{secondface} was given a mathematically rigorous formulation
only quite recently with the advent of 3-dimensional Gaussian multiplicative chaos measure \cite{RVrevis}. 
The Bacry-Muzy construction, which allows the underlying field to be generated by a general infinitely divisible distribution, \emph{i.e.} is not restricted to be gaussian, was motivated by models of turbulent velocity fields that are based on the Poisson and other infinitely divisible 
distributions, cf. \cite{CD95}, \cite{CGH}, \cite{Novikov94}, \cite{She}. The practical interest in non-gaussian fields is that the associated multipliers are no longer lognormal so that the multiscaling spectrum is not restricted to a parabola, 
which adds flexibility to such models, cf. \cite{Chainais}. 


An essential feature of IDMC measures is that they are defined as the exponential functional of a regularized,
\emph{logarithmically correlated}, infinitely divisible field in the limit of zero regularization. In addition to
being a useful tool in stochastic modeling due to their multifractal properties, these measures
are also of interest as a tool for studying the underlying field. Gaussian logarithmically correlated fields, cf. \cite{RVLog}, 
appear in statistical mechanics of disordered energy landscapes \cite{ClD}, 
\cite{Fyo10}, limiting statistics that arise in random matrix theory \cite{FKS}, \cite{FyodSimm}, \cite{Hughes}, 
and even statistics of the Riemann zeta function on the critical line \cite{BK}, \cite{YK}, \cite{Menon}.
One is interested in particular in the distribution of extremes of such fields, which exhibit highly non-classical behavior. 
As it was first discovered in pioneering works \cite{FyoBou} and \cite{FLDR} using methods of statistical physics and followed by \cite{Cao}, \cite{Cao2}, \cite{FLDR2}, \cite{FLDR3}, \cite{Me16}, 
conjectured laws of the total mass of the Gaussian Multiplicative Chaos (GMC) measures on the circle and interval yield precise asymptotic distributions of extremes of the underlying gaussian fields such as restrictions of the 2D Gaussian Free Field to these and similar geometries, for example. The distribution of the maximum of the field was rigorously related to the law of the total mass of the critical GMC in \cite{Madmax}. Further, fluctuations of counting statistics that converge to $H^{1/2}$ gaussian fields can be rigorously quantified by means of the law of the total mass of the corresponding GMC measure, thereby connecting random matrix and GMC theories, cf. \cite{joint}. In addition, the moments of the total mass of the Bacry-Muzy limit measure in 1D are known to be given by the Selberg integral
on the interval, cf. \cite{MRW}, and by the Morris integral on the circle, cf. \cite{FyoBou}, hence connecting
GMC theory with many areas of mathematics, where these integrals occur \cite{ForresterBook}. 
One expects that properties of non-gaussian, logarithmically correlated, infinitely divisible fields 
can be similarly studied by means of the associated IDMC measures. We refer the interested reader to \cite{RV2} for a general review of GMC, to \cite{MeIMRN} and \cite{Menon} 
for detailed reviews of the Bacry-Muzy GMC on the interval and to \cite{Me5} and \cite{MeLMP} for the Bacry-Muzy IDMC  on the interval.



\subsection{What is intermittency differentiation?}
A particularly important open problem in the theory of multiplicative chaos is to compute the distribution of the total mass of the limit measure and, more generally, understand its stochastic dependence structure as in most of the aforementioned 
areas the objects of interest can be reduced to questions about the total mass. The importance of this problem is also
evident from Eq. \eqref{cascade} for it implies that the law of the process on all scales less than 1 is determined by
the law on scale 1, \emph{i.e.} the total mass $M(0,1).$  The primary challenge of this problem is that
the underlying infinitely divisible field has strongly stochastically dependent increments so the usual Markovian techniques that work for
L\'evy processes do not apply. In addition, the limit measure is defined as a limit of the exponential functional of this field 
in the limit of zero regularization, where the field diverges, so one has to deal with a singular limit of a strongly stochastically-dependent 
process. Moreover, the recovery of the total mass distribution from its moments is not possible in the sense of the classical moment problem as 
the moments diverge at any level of intermittency for most multiplicative chaos measures. 

Let $\omega_{\mu, \varepsilon}(u)$ denote the underlying logarithmically correlated field. It is characterized by the
intermittency parameter $\mu$ and its Gaussian component $\sigma^2$ and  L\'evy spectral function $d\mathcal{M}(u),$ 
the latter two parameters describing its infinitely divisible distribution. 
For example, in the gaussian case, $d\mathcal{M}(u)=0,$ the field is determined by its covariance $\textbf{Cov}(
\omega_{\mu, \varepsilon}(u), \,\omega_{\mu, \varepsilon}(v)) \thicksim -\mu \log |u-v|,$
up to $\varepsilon-$dependent regularization near $u=v.$ The limit measure is then defined to be
\begin{equation}
M_{\mu}(a, b)=\lim\limits_{\varepsilon\rightarrow 0}\int\limits_a^b \exp\bigl(\omega_{\mu,\varepsilon}(u)\bigr) \, du.
\end{equation}
The problem is to compute the law of $M_{\mu}(0, 1).$ To overcome the aforementioned 
difficulties for the Bacry-Muzy GMC measure on the interval, we introduced in \cite{Me2} a novel mathematical technique of intermittency 
differentiation that replaced time with intermittency and the non-existent Markov property of the underlying gaussian field
with intermittency invariance. This technique allowed us to derive rules of intermittency differentiation for
the total mass and its dependence structure,
\emph{i.e.} the rules for computing
\begin{subequations}
\begin{align}
\frac{\partial}{\partial \mu} &\textbf{E} [F(M_{\mu}(0, 1)], \label{mudiff}\\
\frac{\partial}{\partial \mu} &\textbf{E} \Bigl[F_1\bigl(M_{\mu}(I_1)\bigr)\,F_2\bigl(M_{\mu}(I_2)\bigr)\Bigr],\; I_1, \,I_2\subset (0,1), \label{mudiffmulti}
\end{align}
\end{subequations}
for general test functions
of the total mass in the form of a non-local Feynman-Kac equation. 
The rule of intermittency differentiation in turn led to the intermittency renormalization
solution to the moment problem, which is based on the special property of GMC,
\begin{equation}
\frac{\partial^n}{\partial \mu^n}\Big\vert_{\mu=0} \textbf{E} \Bigl[\bigl(M_{\mu}(0, 1)-1\bigr)^k\Bigr] = 0,\; k>2n, \label{renormone} 
\end{equation}
and showed how to recover the distribution from the moments
by systematically removing infinity from them, \emph{i.e.} how to reconstruct the Mellin transform $\textbf{E}[M_{\mu}(0, 1)^q],$ $q\in\mathbb{C},$ from the expansion of the positive integer moments in powers of $\mu,$
cf. \cite{Me3}. Finally, in \cite{Me4} we summed the intermittency (high temperature) expansion of the Mellin transform of the total mass 
in closed form 
and proved that the resulting formula is the Mellin transform of a valid positive probability distribution, which we termed 
the Selberg integral probability distribution, see \cite{MeIMRN}, \cite{Me14}
for its detailed analytic and probabilistic analysis, respectively, which is then
naturally conjectured to be the distribution of the total mass. Interestingly, the same formula for the Mellin transform was obtained independently using
a different technique in \cite{FLDR}. It is still an open problem to verify our conjecture, see \cite{Me16} for a review of
the Selberg integral probability distribution and the construction of the Morris integral probability distribution, which is
conjectured, following the work of \cite{FyoBou}, to be the distribution of the total mass of the GMC measure on the circle. We refer the reader to \cite{MeIntRen} for a detailed discussion of these and other related conjectures. 

\subsection{A sumary of key results}
The contribution of this paper is to extend the technique of intermittency differentiation to the general IDMC measure in 1D.\footnote{
Our technique is not limited to 1D. The interested reader can see it applied to multidimensional GMC measures in \cite{MeIntRen}. 
We restrict ourselves to 1D in this paper for simplicity.} 
Our main contribution is the derivation of intermittency differentiation rules
for both the total mass of the IDMC measure, cf. Eq. \eqref{mudiff},  and its dependence structure, cf. Eq. \eqref{mudiffmulti},
in the form
of non-local Feynman-Kac equations. These equations relate the intermittency derivative 
of a class of functionals of the total mass and the L\'evy-Khinchine formula of the underlying infinitely divisible
distribution, \emph{i.e.} $\sigma^2$ and $d\mathcal{M}(u).$ 
We believe our results to be new for all multiplicative chaos measures other than the GMC and 
a first major step towards to the computation of the distribution of the total mass in the non-gaussian case. 
In particular, they allow us to explicitly compute the intermittency derivative of the distribution 
at zero intermittency. Let $M_\mu \triangleq M_\mu(0, 1)$ denote the total mass. 
\begin{align} \frac{\partial}{\partial \mu}\Big\vert_{\mu=0} {\bf E}\bigl[F(M_{\mu})\bigr]  
& = -\sigma^2 F^{(2)}(1)\!\!\int\limits_{\{s_1<s_2\}} 
\log|s_1-s_2|\,ds^{(2)}- \nonumber \\ &- \sum\limits_{k=2}^\infty F^{(k)}(1)\!\!\!\int\limits_{\mathbb{R}\setminus \{0\}} 
(e^u-1)^k\,d\mathcal{M}(u) \!\!\!\!\!\!\!\int\limits_{\{s_1<\cdots<s_k\}} \!\!\!\!\!
\log |s_1-s_k|\,ds^{(k)},
\end{align}
thereby determining the total mass to the first order in intermittency, 
and its covariance structure 
\begin{equation}
{\bf Cov}\Bigl(\log M_\mu(t, t+\tau), \,\log M_\mu(0, \tau)\Bigr) = -\mu \log t\Bigl(\sigma^2 + \int\limits_{\mathbb{R}\setminus \{0\}} u^2 
\,d\mathcal{M}(u)\Bigr) + O(\tau),
\end{equation}
in the limit $\tau\rightarrow 0.$ 
The significance of the latter result is that it allows one to infer the intermittency parameter from empirical data given
the knowledge of the underlying infinitely divisible distribution as was first pointed out in the gaussian case in \cite{MRW}. Our result shows
that the covariance structure of the total mass remains logarithmic in the general infinitely divisible case. 
In addition, we identify a class of non-local functionals of the limit measure
\begin{equation}
v(\mu, F, t_1\cdots t_n)\triangleq \lim\limits_{\varepsilon\rightarrow 0} {\bf E}\Bigl[F\Bigl(\int\limits_0^1 e^{\omega_{\mu,\varepsilon}(u)}\, du\Bigr) 
e^{ \omega_{\mu,\varepsilon}(t_1)+\cdots+\omega_{\mu,\varepsilon}(t_n)}\Bigr],
\end{equation}
that are invariant under intermittency differentiation and derive the differentiation rule for them, which allows the computation
of intermittency derivatives of all orders, \emph{i.e.} the full high temperature (low intermittency) expansion. 
The intermittency derivative of $v(\mu, F, t_1\cdots t_n)$ is given in our most general result, Theorem \ref{IDMC}.
Here we record the special case of $n=0$ to explain the need for considering the functionals $v(\mu, F, \,t_1\cdots t_n)$
in the computation of higher intermittency derivatives. 
\begin{gather}
\frac{\partial}{\partial \mu} {\bf E}\bigl[F(M_{\mu})\bigr]  
 = -\sigma^2 \!\!\!\!\int\limits_{\{s_1<s_2\}} v(\mu, F^{(2)}\!, \, s_1, s_2) 
\log |s_1-s_2|\,ds^{(2)}- \nonumber \\ - \sum\limits_{k=2}^\infty \,\int\limits_{\mathbb{R}\setminus \{0\}} 
(e^u-1)^k\,d\mathcal{M}(u) \int\limits_{\{s_1<\cdots<s_k\}} \!\!\!\!\!\!\!\!\!\!\!\!
v(\mu, F^{(k)}\!,\, s_1\cdots s_k) \log |s_1-s_k|\,ds^{(k)}. \label{firstderiv} 
\end{gather}
 
Our derivation of the intermittency differentiation rule is exact in the sense of equality of formal power series but not 
mathematically rigorous as we shun all questions of convergence while operating with infinite series. The rest of our results 
are rigorous. In particular, we give a rigorous derivation of the key combinatorial property of the underlying infinitely divisible field in Lemma \ref{BLemma}, which is responsible for the terms on the right-hand side of Eq. \eqref{firstderiv} and, in particular, explains why GMC is unique among all the other IDMC measures. 
The lemma  says that for any continuous function $\mathfrak{f}(\delta,s)$ vanishing as $\delta\rightarrow 0$ and 
$s_1<\cdots <s_n$ we have the estimate
\begin{subequations}
\begin{align}
{\bf E} \Bigl[\prod_{j=1}^2 e^{\mathfrak{f}(\delta,s_j)+\omega_{\mu, \varepsilon}(s_j)} - 1 \Bigr] = &
\prod_{j=1}^2 \bigl(e^{\mathfrak{f}(\delta,s_j)}-1\bigr) - \delta \log |s_2-s_1| \times \nonumber \\ & \times \Bigl(\sigma^2+ \int\limits_{\mathbb{R}\setminus \{0\}}  (e^u-1)^2 d\mathcal{M}(u)\Bigr)+
o(\delta),\\
{\bf E} \Bigl[\prod_{j=1}^n e^{\mathfrak{f}(\delta,s_j)+\omega_{\mu, \varepsilon}(s_j)} - 1 \Bigr] = &
\prod_{j=1}^n \bigl(e^{\mathfrak{f}(\delta,s_j)}-1\bigr) -
\delta \log |s_n-s_1| \times \nonumber \\ &\times \int\limits_{\mathbb{R}\setminus \{0\}}  (e^u-1)^n d\mathcal{M}(u)+
o(\delta),
\end{align}
\end{subequations}
for $n>2$ so that the GMC intermittency derivative only requires the second derivative of the test function, while 
the general IDMC intermittency derivative requires the second and all higher derivatives, cf. Eq. \eqref{firstderiv}. 
We also give a rigorous derivation of  
the differentiation rule for positive integer moments of the total mass, \emph{i.e.} we prove Eq. \eqref{firstderiv} for $F(x)=x^n,$ $n\in \mathbb{N},$ in Theorem \ref{intidentity} thereby verifying
our main result in this case. While these results were first established in the gaussian setting in \cite{Me2} and \cite{Me3}, their extension to the general infinitely divisible setting is quite non-trivial and constitutes the primary technical innovation of the paper. 

The structure of the paper is as follows. In Section 2 we describe the key properties of the Bacry-Muzy construction. 
In Section 3 we state all of our results. In Section 4 we give the proofs. Section 5
concludes. In the Appendix we give the more technical details of the Bacry-Muzy construction and extend it to a general intensity measure.

\section{A survey of IDMC chaos}
\noindent In this section we will review the infinitely divisible multiplicative chaos (IDMC) construction of Bacry and Muzy \cite{BM1} and \cite{BM}, including the formula for positive integer moments of the total mass that we first noted in \cite{MeLMP}. The Bacry-Muzy construction of IDMC measure on the unit interval is based on the idea
of exponentiating a regularized, infinitely divisible (ID), logarithmically correlated field in the limit of zero regularization. The details of how such fields are constructed are very technical and are relegated to the Appendix. Here
we focus instead on the most basic properties of the construction. 

Let us the denote the ID field by $\omega_{\mu, \varepsilon}(u).$ $\mu>0$ is the intermittency parameter,
it is a fundamental constant of the construction. Its range needs to be restricted for the limit measure to exist as explained below.
$\varepsilon$ is the regularization scale. It is well-known that an infinitely divisible probability distribution
is uniquely characterized by its L\'evy-Khinchine representation, cf. \cite{Steutel}. Let us write the
characteristic function of $\omega_{\mu, \varepsilon}(u)$ 
in the  form
\begin{equation}
{\bf{E}} \left[ e^{i q \omega_{\mu, \varepsilon}(u)} \right]=e^{-\mu\phi(q)\log\varepsilon},
\,\,\,q\in\mathbb{R},
\end{equation}
where the  L\'evy-Khinchine representation of $\phi(q)$ is 
\begin{equation}\label{phi}
\phi(q) = -\frac{iq\sigma^2}{2}  - \frac{q^2\sigma^2}{2} + 
\int\limits_{\mathbb{R}\setminus \{0\}} \Bigl(e^{iq u}-1-iq(e^u-1)
\Bigr) d\mathcal{M}(u).
\end{equation}
Note that it is normalized by $\phi(-i)=0,$ which is required by Kahane's theory \cite{K2} for the limit measure to exist. 
When the spectral function $\mathcal{M}(u)=0,$ the distribution is gaussian. 
It must be emphasized that $\omega_{\mu, \varepsilon}(u)$ is not a L\'evy process in $u,$ rather its increments are
stochastically dependent and the structure of their dependence is quite non-trivial and determined by Lemma \ref{main} 
given in the Appendix. Here we state the simplest non-trivial case of this dependence. Let $u_1\leq u_2$ and $q_1,\,q_2\in\mathbb{R}.$
\begin{align}
{\bf E}\Bigl[e^{\bigl(i q_1 \omega_{\mu,\varepsilon}(u_1)+
i q_2 \omega_{\mu,\varepsilon}(u_2)\bigr)
}\Bigr] = & e^{-\mu\log\varepsilon\bigl(\phi(q_1)+\phi(q_2)\bigr)}\times\nonumber \\  &\times e^{\mu\bigl(\phi(q_1+q_2)-\phi(q_1)-\phi(q_2)\bigr)\rho_\varepsilon(u_2-u_1)},
\end{align}
where the function $\rho_\varepsilon(u),$ known as the intensity measure of the construction,\footnote{What we call
$\mu$ and $\rho$ is denoted by $\lambda^2$  and $\mu,$ respectively, in \cite{BM}.} is
\begin{equation}\label{rho}
\rho_{\varepsilon}(u) =
\begin{cases}
-\log |u|& \, \text{if $\varepsilon\leq |u|\leq 1$}, \\
- \log \varepsilon + \bigl(1-\frac{|u|}{\varepsilon}\bigr) & \, \text{if
$|u|<\varepsilon$},
\end{cases}
\end{equation}
and it is identically zero for $|u|>1.$ One observes that the process $\omega_{\mu, \varepsilon}(u)$ is a \emph{logarithmically correlated} ID field.
This is a fundamental property of the Bacry-Muzy construction and of all known
multiplicative chaos constructions. The idea of using a logarithmically correlated field
in the construction of a multifractal process was introduced by Mandelbrot \cite{secondface}
in his attempt to give a mathematical formulation of the Kolmogorov-Obukhov \cite{Ko}, \cite{Obu} theory
of energy dissipation in developed turbulence by means of a precursor of the modern GMC theory. 

The IDMC measure is defined as the limit of the exponential functional of the field $\omega_{\mu, \varepsilon}(u)$
in the limit of zero regularization
\begin{equation}
M_{\mu}(a, b)=\lim\limits_{\varepsilon\rightarrow 0}\int\limits_a^b \exp\bigl(\omega_{\mu,\varepsilon}(u)\bigr) \, du.
\end{equation}
Its existence is a deep result as the field $\omega_{\mu, \varepsilon}(u)$ diverges in this limit so that the order of
integral and limit cannot be interchanged. It turns out that the limit exists only for a finite range of $\mu,$ 
and this range is determined by the condition\footnote{The non-degeneracy condition given in \cite{BM1} is
less stringent than Eq. \eqref{nondeg}, which is however sufficient in most
cases of interest such as those of the limit lognormal and 
Poisson measures.}
\begin{equation}\label{nondeg}
1+i\mu\phi'(-i) = 1-\mu\Bigl(\frac{\sigma^2}{2} + \int\limits_{\mathbb{R}\setminus \{0\}} \bigl(ue^{u}-e^u+1
\bigr) d\mathcal{M}(u)\Bigr)
>0.
\end{equation}
The limit measure has the stationarity property
\begin{equation}\label{station}
M_{\mu}(t,t+\tau) \overset{{\rm in \,law}}{=} M_{\mu}(0,\tau)
\end{equation}
and is 
non-degenerate in the sense of ${\bf E}[M_{\mu}(a, b)]=|b-a|.$

The self-similarity property of the limit measure follows
the logarithmic correlation property of the underlying field and the specific functional
form of the intensity measure in Eq. \eqref{rho}. The details are given in the Appendix.
Here we record the law of the multiplier in Eq. \eqref{cascade},
\begin{equation}
\textbf{E}[e^{iq \log W_\gamma}] = \gamma^{iq-\mu\phi(q)},\; \gamma <1,
\end{equation}
which is log-infinitely divisible as expected.
The resulting scaling law of the moments is
\begin{equation}\label{qscaling}
{\bf E}\bigl[M_{\mu}(0, t)^q\bigr] = const(q)\,
t^{q-\mu\phi(-iq)}, \, t<1.
\end{equation}
Hence, $q-\mu\phi(-iq)$ is the multiscaling spectrum of
the limit measure. We have
\begin{equation}\label{mspectrum}
q-\mu\phi(-iq) = q-\mu\Bigl(\frac{\sigma^2}{2}(q^2-q) + \int\limits_{\mathbb{R}\setminus \{0\}} \bigl(e^{qu}-1-q(e^u-1)
\bigr) d\mathcal{M}(u)\Bigr).
\end{equation}
This function also controls the moments. 
The moments $q>1$ of $M_{\mu}(0, t)$ are finite under the
following necessary and sufficient conditions,
\begin{subequations}
\begin{align}
& q-\mu\phi(-iq)>1 \Longrightarrow {\bf E}[M^q_{\mu}(0, t)]<\infty, \label{finmom} \\
& {\bf E}[M^q_{\mu}(0, t)]<\infty \Longrightarrow
q-\mu\phi(-iq)\geq 1.
\end{align}
\end{subequations}

Finally, we can give an explicit multiple integral representation for the moments of the limit measure.
Given $m\in\mathbb{N},$  define the quantity $d(m)$ by
\begin{equation}\label{d}
d(m)\triangleq \sigma^2+\int\limits_{\mathbb{R}\setminus \{0\}} e^{(m-1)u}
(e^u-1)^2\,d\mathcal{M}(u).
\end{equation}
Then, the $n$th moment of the total mass is given by a generalized Selberg integral of dimension $n.$
Let $0\leq a<b\leq 1$ and $n$ satisfy Eq. \eqref{finmom}. 
\begin{equation}\label{singlemomformula}
{\bf E}\Bigl[\Bigl(\int\limits_a^b M_\mu(dt)\Bigr)^n\Bigr] = n!
\int\limits_{\{a<t_1<\cdots<t_n<b\}} \prod\limits_{k<p}^n |t_p-t_k|^{-\mu\,d(p-k)}
\,dt^{(n)}.
\end{equation}
This result is due to \cite{MRW} in the canonical gaussian case $(d\mathcal{M}(u)=0)$ and to \cite{MeLMP} in general. 

Throughout the rest of the paper we assume that the non-degeneracy condition in Eq. \eqref{nondeg} is satisfied,
\emph{i.e.} we work in the co-called sub-critical regime. We also adapt the following slight abuse of terminology 
by referring to the limit measure $dM_\mu$ as an IDMC measure if $d\mathcal{M}(u)\neq 0$ 
and GMC measure if $d\mathcal{M}(u)=0$ in Eq. \eqref{phi}. In the GMC case we let $\sigma=1.$

We note that the canonical Bacry-Muzy construction can be somewhat extended by allowing a more general
intensity measure than what was given in Eq. \eqref{rho}, cf. Eq. \eqref{rhogeneral}. The resulting construction is summarized in the Appendix. 



We end this section with two main examples of IDMC measures: gaussian (GMC) 
and Poisson.
\subsection{Limit lognormal measure}
Let $\sigma=1,$ $\mathcal {M}(u)=0$ in Eq. \eqref{phi}. Then, the moments are given by the classical
Selberg integral,
\begin{equation}\label{singleLlog}
{\bf E}\bigl[ M_\mu(0,\,1)^n\bigr] = n!
\int\limits_{0<t_1<\cdots<t_n<1} \prod\limits_{k<p}^n |t_p-t_k|^{-\mu}
\,dt^{(n)}.
\end{equation}
Note that the non-degeneracy condition in Eq. \eqref{nondeg} amounts to
$0<\mu<2$ and that the moments become infinite for $n>2/\mu.$ 
The process $\omega_{\mu, \varepsilon}(u)$ in this case has covariance $-\mu\log|u-v|,$
\emph{i.e.}  represents an ideal $1/f$ noise and can be thought of as the
restriction of the 2D gaussian free field to the unit interval. It was first introduced in \cite{MRW}.
The multiscaling spectrum of the limit measure in Eq. \eqref{mspectrum} is the parabola
\begin{equation}
q-\mu\phi(-iq) = q-\frac{1}{2}\mu q (q-1).
\end{equation}
$\log W_\gamma$ is gaussian with mean $(1+\mu/2)\log\gamma$ and variance $-\mu\log\gamma.$


\subsection{Limit Log-Poisson measure}
Let $\sigma=0$ and $d\mathcal{M}(u) = \delta\bigl(u-log(c)\bigr)du$ in Eq. \eqref{phi}, \emph{i.e.}
the underlying distribution is a point mass at $\log(c),$ $c>0,$ $c\neq 1.$
\begin{equation}\label{P}
{\bf E}\bigl[ M_\mu(0,\,1)^n\bigr] = n!
\int\limits_{0<t_1<\cdots<t_n<1} \prod\limits_{k<p}^n |t_p-t_k|^{-\mu (c-1)^2 c^{p-k-1}}
\,dt^{(n)}.
\end{equation}
The non-degeneracy condition in Eq. \eqref{nondeg} is
\begin{equation}\label{nondegP}
0<\mu<\frac{1}{c\log(c)-c+1},
\end{equation}
so that the limit log-Poisson measure exists for any such $c$ as $c\log(c)-c+1>0$
for $c>0,$ $c\neq 1.$ The moments are finite for $q>1$ if the multiscaling spectrum satisfies
\begin{equation}\label{momentexistP}
q-\mu\phi(-iq) = q-\mu\bigl(c^q-1-q(c-1)\bigr)>1, 
\end{equation}
cf. Eqs.  \eqref{mspectrum} and \eqref{finmom}.
In particular, the moments become eventually infinite if
$c>1$ as they do in the limit lognormal case. On the contrary, if $0<c<1,$
all moments for $q>1$ are finite for sufficiently small $\mu.$
The process $\omega_{\mu, \varepsilon}(u)$ in this case was first constructed in \cite{Pulses}
and the type of spectrum in Eq. \eqref{momentexistP} appeared in \cite{She}.

 





\section{Intermittency invariance and differentiation rule}
\noindent
In this section we will formulate the intermittency invariance of the underlying infinitely divisible (ID) field and state
our main results on the rule of intermittency differentiation and its application to the distribution of the total mass
at the lowest non-trivial order in intermittency. The proofs are deferred to Section 4. The reader
can assume with little loss of generality that the intensity measure is the measure of Bacry-Muzy given in Eq. \eqref{rho}.
All of our results work with the general intensity measure given in Eq. \eqref{rhogeneral}.  

Fix $L\geq 1$ and define the ID random variable by the formula
\begin{equation}
{\bf E}[e^{iq Z_L}] = e^{\mu \phi(q) \log L}.
\end{equation}
Define the corresponding ID field by
\begin{equation}\label{omegaLprocess}
\omega_{\mu,L,\varepsilon}(u) = \omega_{\mu,\varepsilon}(u) + Z_L,
\end{equation}
where $Z_L$ is independent of the process $\omega_{\mu,\varepsilon}(u).$ Clearly, $\omega_{\mu,L=1,\varepsilon}(u)$ coincides
with the original field. 
Finally, let
$\delta\rightarrow X(\delta)$ be a L\'evy process (a stochastic process with
stationary, independent increments)
that is independent of the $u\rightarrow \omega_{\mu, L,
\varepsilon}(u)$ process and defined in terms of the ID distribution
associated with $\phi(q)$ as follows
\begin{equation}\label{Xdelta}
{\bf E}\left[e^{i q X(\delta)}\right] = e^{\delta\phi(q)}, \,\,
X(0)=0.
\end{equation}
The existence and uniqueness of $X(\delta)$ follow from the general
theory of L\'evy processes, confer \cite{Bertoin}. Then, we have the following
result.
\begin{thm}[Intermittency invariance]\label{IntInv}
Fix $\mu,$ $L,$ $\varepsilon,$ and $\delta<\mu,$ and let
$\bar{\omega}_{\delta, e L, \varepsilon}(u)$ denote an independent
copy of the $\omega_{\mu,L,\varepsilon}(u)$ process with the intermittency $\delta$ and $L$ replaced with $eL,$ where $e$ is the base of the
natural logarithm. Then, there holds the
following equality in law
of stochastic processes in $u$ on the interval $u\in[0, 1],$
\begin{gather}\label{theinvariance}
X(\delta)+\omega_{\mu, L, \varepsilon}(u) = \omega_{\mu-\delta, L,
\varepsilon}(u) + \bar{\omega}_{\delta, e L, \varepsilon}(u).
\end{gather}
\end{thm}
In the gaussian case this result is originally due to  \cite{Me2}, \cite{Me3} 
and to \cite{MeIntRen} for the general intensity measure. 
In the ID case this result is due to  \cite{Me5} for the Bacry-Muzy measure 
and its  extension to the general case is new. 

The significance of the intermittency invariance is that it provides a technical devise that 
replaces the non-existent Markov property of the underlying ID field and allows one to
derive a Feynman-Kac equation for the distribution of the total mass by 
considering a stochastic flow in intermittency as opposed to time (as in the classical framework
of diffusions). 

Define the finite regularization scale total mass to be
\begin{equation}\label{Meps}
M_{\mu, \varepsilon}\triangleq \int_0^1 e^{ \omega_{\mu,\varepsilon}(s)} \, ds
\end{equation}
so that the total mass is the limit 
\begin{equation}
M_{\mu}= \lim\limits_{\varepsilon\rightarrow
0} M_{\mu, \varepsilon}.
\end{equation}
Our results are most naturally expressed in terms of a particular \emph{non-local} functional of the total mass of the form
\begin{equation}\label{generfunctional}
v(\mu, F, t_1\cdots t_n)\triangleq \lim\limits_{\varepsilon\rightarrow 0} {\bf E}\Bigl[F\bigl(M_{\mu, \varepsilon}\bigr) 
e^{ \omega_{\mu,\varepsilon}(t_1)+\cdots+\omega_{\mu,\varepsilon}(t_n)}\Bigr],
\end{equation}
involving the entire path of $dM_\mu$ as opposed to the value of the total mass of the whole interval. As we will see below, 
the functionals in Eq. \eqref{generfunctional} are in fact invariant under intermittency differentiation. They
are well-defined for sufficiently small intermittency due to the normalization condition $\phi(-i)=0,$ 
cf. Eq. \eqref{phi}, which implies
\begin{equation}
{\bf E}\bigl[e^{\omega_{\mu,\varepsilon}(s)}\bigr]=1
\end{equation}
so that the functional in Eq. \eqref{generfunctional} can be naturally interpreted as a change of probability measure.
Define also the key quantity
\begin{equation}\label{gfunction}
g(s_1, \,s_2) \triangleq \lim\limits_{\varepsilon\rightarrow 0} \rho_\varepsilon(s_1-s_2),
\end{equation}
which is the limit of the intensity measure,  
cf. Eqs. \eqref{rho} for the Bacry-Muzy and \eqref{rhogeneral} for the general 
case. In the Bacry-Muzy case, one has
\begin{equation}
g(s_1, \,s_2) = -\log |s_1-s_2|.
\end{equation}
Then, given a smooth test function $F(x),$ our main results are as follows. 
\begin{thm}[Rule of intermittency differentiation]\label{theoremIDdiff}
\begin{gather} \frac{\partial}{\partial \mu} {\bf E}\bigl[F(M_{\mu})\bigr]  
 = \sigma^2 \!\!\!\!\int\limits_{\{s_1<s_2\}} v(\mu, F^{(2)}\!, \,s_1, s_2)\,
g(s_1,\,s_2)\,ds^{(2)}+ \nonumber \\ + \sum\limits_{k=2}^\infty \,\int\limits_{\mathbb{R}\setminus \{0\}} 
(e^u-1)^k\,d\mathcal{M}(u) \!\!\!\!\!\int\limits_{\{s_1<\cdots<s_k\}} \!\!\!\!\!\!\!
v(\mu, F^{(k)}\!, s_1,\cdots s_k) \,g(s_1,\,s_k)\,ds^{(k)}.\label{theIDrule}
\end{gather}
\end{thm}
In the special case of the GMC this result appeared first in \cite{Me2} and \cite{Me3}
in the Bacry-Muzy case and in \cite{MeIntRen} in general. 
In the ID case this result is new in all cases. A derivation of Theorem \ref{theoremIDdiff} from Theorem \ref{IntInv} is given in Section 4. It suffices to explain here that the main idea is to consider a stochastic flow in intermittency
and evaluate the limit
\begin{equation}\label{thedeltalimit}
\frac{\partial}{\partial\delta}\Big\vert_{\delta=0}\,\,
{\bf{E}} \Bigl[ F\bigl(ze^{X(\delta)} \,M_{\mu, \varepsilon}\bigr)\Bigr],
\end{equation}
where $X(\delta)$ is defined in Eq. \eqref{Xdelta} and is independent of $\omega_{\mu,\varepsilon}(s),$  in two different ways: 
by the backward Kolmogorov equation for the L\'evy process $X(\delta)$ and by applying Theorem \ref{IntInv} and expanding to the first order in $\delta.$

Upon substituting $\mu=0$ into Eq. \eqref{theIDrule}
we obtain an explicit formula for the first order term in the expansion of the distribution of the total mass in intermittency.
\begin{cor}[Distribution to the first order in intermittency]
The distribution of the total mass to the first order in intermittency is determined by
\begin{align} \frac{\partial}{\partial \mu}\Big\vert_{\mu=0} {\bf E}\bigl[F(M_{\mu})\bigr]  
& = \sigma^2 F^{(2)}(1)\int\limits_{\{s_1<s_2\}} 
g(s_1,\,s_2)\,ds^{(2)}+ \nonumber \\ &+ \sum\limits_{k=2}^\infty F^{(k)}(1)\,\int\limits_{\mathbb{R}\setminus \{0\}} 
(e^u-1)^k\,d\mathcal{M}(u) \!\!\!\!\!\!\!\int\limits_{\{s_1<\cdots<s_k\}} \!\!\!
g(s_1,\,s_k)\,ds^{(k)}.\label{theIDruleFirst}
\end{align}
\end{cor}
\begin{rem}
It is easy to show that 
\begin{equation}
\int\limits_{\{s_1<\cdots<s_k\}} 
g(s_1,\,s_k)\,ds^{(k)} = O\bigl(1/k!\bigr)
\end{equation}
in the limit $k\rightarrow \infty$ so that the sum in Eq. \eqref{theIDruleFirst} is finite provided
\begin{equation}
\int\limits_{\mathbb{R}\setminus \{0\}}  F(e^u) \,d\mathcal{M}(u)<\infty.
\end{equation}
\end{rem}

The technique of intermittency differentiation is not limited to the distribution of the total mass
of the limit measure but applies also to the joint distribution of the measure of subintervals, \emph{i.e.}
the dependence structure of the measure. We will illustrate this application with the case of two
disjoint subintervals $I_1, I_2\subset[0,1],$ $I_1\cap I_2=\emptyset,$  $\sup I_1<\inf I_2$
for simplicity, although it applies to any finite number of such subintervals.
Denote
\begin{equation}\label{Mepsinterval}
M_{\mu, \varepsilon}(I) \triangleq \int_{I} e^{ \omega_{\mu,\varepsilon}(s)} \, ds
\end{equation}
and the limit measure of the interval
\begin{equation}
M_{\mu}(I)= \lim\limits_{\varepsilon\rightarrow
0} M_{\mu, \varepsilon}(I).
\end{equation}
To simplify notations, it is also convenient to introduce the functional
\begin{equation}\label{generfunctionaltwo}
v(\mu, F_1, F_2, t_1\cdots t_n)\triangleq \lim\limits_{\varepsilon\rightarrow 0} {\bf E}\Bigl[F_1\bigl(M_{\mu, \varepsilon}(I_1)\bigr) 
F_2\bigl(M_{\mu, \varepsilon}(I_2)\bigr) 
e^{ \omega_{\mu,\varepsilon}(t_1)+\cdots+\omega_{\mu,\varepsilon}(t_n)}\Bigr],
\end{equation}
extending Eq. \eqref{generfunctional} to two subintervals, which are assumed to be fixed and dropped from the list of arguments. 

\begin{thm}[Intermittency differentiation for two intervals]\label{theoremIDdiffTwo}
\begin{gather} \frac{\partial}{\partial \mu} {\bf E}\Bigl[F_1\bigl(M_{\mu}(I_1)\bigr) F_2\bigl(M_{\mu}(I_2)\bigr)\Bigr]  
 = \sigma^2\Biggl[\, \int\limits_{\{s_1<s_2\}\subset I_1}  v(\mu, F_1^{(2)}\!, F_2, \,s_1, s_2)   \times \nonumber \\
 \times 
g(s_1,\,s_2)\,ds^{(2)} +  \int\limits_{\{s_1<s_2\}\subset I_2} v(\mu, F_1,  F_2^{(2)}\!, \,s_1, s_2)\,
g(s_1,\,s_2)\,ds^{(2)} + \nonumber \\ 
+ \int\limits_{\{s_1\in I_1,\, s_2\in I_2\}} v(\mu, F_1^{(1)}\!,  F_2^{(1)}\!, \,  s_1, s_2) \,
g(s_1,\,s_2)\,ds^{(2)}\Biggr] +
 \nonumber \\ + 
\sum\limits_{\substack{k,l\geq 0 \\ k+l\geq 2}}^\infty \;\int\limits_{\mathbb{R}\setminus \{0\}} 
(e^u-1)^{k+l}\,d\mathcal{M}(u) \!\!\!\!\! \int\limits_{\substack{\{s_1<\cdots<s_k\}\subset I_1 \\ \{s_{k+1}<\cdots<s_{k+l}\}\subset I_2 }} \!\!\!\!\!\!\!\!\!\!\!\!
v(\mu, F_1^{(k)}\!, F_2^{(l)}\!, \, s_1, \cdots, s_{k+l}) \times \nonumber \\ \times
g(s_1,\,s_{k+l})\,ds^{(k+l)}.\label{theIDruleInterval}
\end{gather}
\end{thm}
This result in the gaussian case is originally due to \cite{Me6}. In the infinitely divisible case it is new.
\begin{cor}[Joint distribution to the first order in intermittency]\label{TwoIntervalsZeroInter}
\begin{gather} \frac{\partial}{\partial \mu}\Big\vert_{\mu=0} {\bf E}\Bigl[F_1\bigl(M_{\mu}(I_1)\bigr) F_2\bigl(M_{\mu}(I_2)\bigr)\Bigr]   
 = \sigma^2\Biggl[ F_1^{(2)}\bigl(|I_1|\bigr)  F_2\bigl(|I_2|\bigr) \times \nonumber \\
 \times \int\limits_{\{s_1<s_2\}\subset I_1}  
g(s_1,\,s_2)\,ds^{(2)} +  F_1\bigl(|I_1|\bigr)  F_2^{(2)}\bigl(|I_2|\bigr)   \int\limits_{\{s_1<s_2\}\subset I_2} 
g(s_1,\,s_2)\,ds^{(2)}
+ \nonumber \\  + F_1^{(1)}\bigl(|I_1|\bigr) F_2^{(1)}\bigl(|I_2|\bigr)  \int\limits_{\{s_1\in I_1,\, s_2\in I_2\}}
g(s_1,\,s_2)\,ds^{(2)}\Biggr] +
 \nonumber \\ + 
\sum\limits_{\substack{k,l\geq 0 \\ k+l\geq 2}}^\infty F_1^{(k)}\bigl(|I_1|\bigr)  F_2^{(l)}\bigl(|I_2|\bigr)
 \int\limits_{\mathbb{R}\setminus \{0\}} 
(e^u-1)^{k+l}\,d\mathcal{M}(u) \!\!\!\!\!\!\!\!\!\!\!\! \!\!\int\limits_{\substack{\{s_1<\cdots<s_k\}\subset I_1 \\ \{s_{k+1}<\cdots<s_{k+l}\}\subset I_2 }} \!\!\!\!\!\!\!\!\!\!\!\!\!\!\!
g(s_1,\,s_{k+l})\,ds^{(k+l)}.\label{theIDruleIntervalZero}
\end{gather}
\end{cor}
This result has an application to the problem of computing the covariance structure of the total mass distribution.
\begin{cor}[Covariance structure]\label{covID}
Let $0<t<1.$ Then, in the limit $\tau\rightarrow 0,$ 
\begin{equation}\label{covstrucID}
{\bf Cov}\Bigl(\log M_\mu(t, t+\tau), \,\log  M_\mu(0, \tau)\Bigr) = \mu g(0,t)\Bigl(\sigma^2 + \int\limits_{\mathbb{R}\setminus \{0\}} u^2 
\,d\mathcal{M}(u)\Bigr) + O(\tau).
\end{equation}
\end{cor}
This result in the gaussian case was originally established in \cite{MRW} for the Bacry-Muzy measure and extended to the infinitely divisible case in \cite{MeLMP}.
Both calculations relied on a heuristic analytic continuation of joint integer moments. Corollary \ref{TwoIntervalsZeroInter} allows us to dispense with
the analytic continuation, as shown in Section 4.   

We will illustrate the rule of intermittency differentiation with the example of positive integer moments. 
Recall the formula for the moments in Eq. \eqref{rsinglemomformula}, assuming $n$  satisfies Eq. \eqref{finmom}.
The special case of the Bacry-Muzy measure is given in Eq. \eqref{singlemomformula} and corresponds to $r(t)=|t|.$
The intermittency derivative is then
\begin{equation}\label{derivsinglemomformula}
\frac{\partial}{\partial \mu} {\bf E}\bigl[M_\mu^n\bigr] 
= n! \!\!\!\!
\int\limits_{\{0<t_1<\cdots<t_n<1\}} \sum\limits_{k<p}^n d(p-k)\,g(t_p, t_k)
\prod\limits_{k<p}^n r(t_p-t_k)^{-\mu\,d(p-k)}
\,dt^{(n)}.
\end{equation}
On the other hand, given $F(x)=x^n,$ Theorem \ref{theoremIDdiff} gives us the formula
\begin{align} \frac{\partial}{\partial \mu} {\bf E}\bigl[M_{\mu}^n\bigr]  
& = \sigma^2 n(n-1) \!\!\!\!  \int\limits_{\{s_1<s_2\}}\!\!\! \lim\limits_{\varepsilon\rightarrow
0} {\bf E}\Bigl[M_{\mu, \varepsilon}^{n-2} \,
e^{ \omega_{\mu,\varepsilon}(s_1)+\omega_{\mu,\varepsilon}(s_2)}\Bigr] 
g(s_1,\,s_2)\,ds^{(2)}+ \nonumber \\ &+ \sum\limits_{k=2}^n \,\int\limits_{\mathbb{R}\setminus \{0\}}  \frac{n!}{(n-k)!}
(e^u-1)^k\,d\mathcal{M}(u) \times \nonumber \\ & \times  \int\limits_{\{s_1<\cdots<s_k\}} 
\lim\limits_{\varepsilon\rightarrow 0} {\bf E}\Bigl[M_{\mu, \varepsilon}^{n-k} \,
e^{ \omega_{\mu,\varepsilon}(s_1)+\cdots+\omega_{\mu,\varepsilon}(s_k)}\Bigr] 
g(s_1,\,s_k)\,ds^{(k)}.\label{theIDruleMom}
\end{align}
The equivalence of Eqs. \eqref{derivsinglemomformula} and \eqref{theIDruleMom} is a corollary of 
the following general integral identity.
\begin{thm}[Intermittency derivative of integer moments]\label{intidentity}
Let $\omega(s)$ and $g(s,\,t)$ be continuous functions and $k=2,\,\cdots, n.$ 
The identity
\begin{gather}
\frac{1}{(n-k)!}  \Bigl(\int_0^1 e^{\omega(s)} ds\Bigr)^{n-k}  \!\!\!\!\!\!\!
\int\limits_{\{s_1<\cdots<s_k\}} \!\!\!
e^{ \omega(s_1)+\cdots+\omega(s_k)}
g(s_1,\,s_k)\,ds^{(k)}   =  \nonumber \\
  \int\limits_{\{s_1<\cdots<s_n\}}  
e^{ \omega(s_1)+\cdots+\omega(s_n)}  \Bigl[\sum\limits_{\substack{i<j\\j-i\geq k-1}}^n 
\binom{j-i-1}{k-2} g(s_i, s_j)\Bigr] ds^{(n)}, \label{intidentityeq}
\end{gather}
implies the equality of the right-hand sides of Eqs. \eqref{derivsinglemomformula} and \eqref{theIDruleMom}. 
\end{thm}
Its proof is given in Section 4. It is worth pointing out that the equivalence of
Eqs. \eqref{derivsinglemomformula} and \eqref{theIDruleMom} in the special case of GMC follows from
the case of $k=2$ of this identity. 
The general case is new and significantly more involved.  

We will conclude this section with a discussion of higher intermittency derivatives, which one wants to compute
to derive the full high temperature (low intermittency)\footnote{The intermittency parameter $\mu$ is usually expressed in terms of the inverse temperature $\beta$ in the form $\mu=2\beta^2$ in the statistical physics literature.} expansion, as we did for the Mellin transform
of the total mass of the GMC measure in \cite{Me4}, \cite{MeIMRN}, \cite{MeIntRen}.
It is clear from the structure of the first intermittency derivative in Theorem \ref{theoremIDdiff} that
in order to compute higher derivatives, \emph{i.e.} to apply the differentiation rule iteratively, one 
needs to establish a differentiation rule for the non-local functionals of the total mass 
that we introduced in Eq. \eqref{generfunctional}.
The functional in Eq. \eqref{generfunctional} can be naturally interpreted as a change of probability measure.
In the case of GMC this change of measure is known explicitly, cf. \cite{Me2} and \cite{Me3}, and
the functional in Eq. \eqref{generfunctional} is a simple change of drift.
In fact, one has the identity,
\begin{align}
 {\bf E}\Bigl[F\bigl(M_{\mu, \varepsilon}\bigr) 
e^{ \omega_{\mu,\varepsilon}(t_1)+\cdots+\omega_{\mu,\varepsilon}(t_n)}\Bigr] = & \exp\Bigl(\mu\sum\limits_{i<j}^n 
\rho_\varepsilon(t_j-t_i)\Bigr) \times \nonumber \\ & \times  {\bf E}\Bigl[F\Bigl(\int_0^1 e^{ \omega_{\mu,\varepsilon}(s)+\mu\sum_{j=1}^n \rho_\varepsilon(s - t_j)} \, ds\Bigr)\Bigr], \label{GirsanovGaussian}
\end{align}
so that 
\begin{align}
v(\mu, F, \, t_1,\cdots t_n) 
= & \exp\Bigl(\mu\sum\limits_{i<j}^n 
g(t_i,\,t_j)\Bigr) {\bf E}\Bigl[F\Bigl(\int_0^1 e^{ \mu\sum_{j=1}^n g(s, t_j)} \, dM_\mu(s)\Bigr)\Bigr], \label{GirsanovGaussianLim}
\end{align}
which is manifestly non-local.  The functional on the right-hand side of Eq. \eqref{GirsanovGaussianLim} is of the form
\begin{equation}\label{thefunctional}
v(\mu,\,f,\,F) \triangleq {\bf E}\Bigl[F\Bigl(\int_0^1 e^{\mu
f(s)}\,M_\mu(ds)\Bigr)\Bigr],
\end{equation}
where $f(s)$ in our case equals
\begin{equation}
f(s) = \sum_{j=1}^n g(s, t_j).
\end{equation}
The intermittency differentiation rule in the gaussian case for a general $f(s)$ is, cf. \cite{Me3}, \cite{MeIMRN}, and \cite{MeIntRen},
\begin{align} \frac{\partial}{\partial \mu} v(\mu,\,
f,\,F) & = \int\limits_{[0,\,1]} v\bigl(\mu,\,f+g(\cdot,
s),\,F^{(1)}\bigr) e^{\mu f(s)}f(s)\,ds+ \nonumber \\
& +\!\!\!\!\!\int\limits_{\{s_1<s_2\}} \!\!\! v\bigl(\mu,
f+g(\cdot,s_1)+g(\cdot,s_2),F^{(2)}\bigr)
e^{\mu\bigl(f(s_1)+f(s_2)+g(s_1,s_2)\bigr)} \times \nonumber \\ & \times g(s_1,
s_2)\,d s^{(2)}.\label{therule}
\end{align}
Hence, applying it to the functional in Eq. \eqref{generfunctional} and using Eq. \eqref{GirsanovGaussianLim}, we obtain
the desired rule of differentiation that allows one to compute intermittency derivatives of all orders in the gaussian case.
\begin{thm}[GMC intermittency differentiation]\label{GMC}
\begin{align}
\frac{\partial}{\partial \mu} v(\mu,\,F, \,t_1\cdots t_n)  & =  v(\mu,\,F, \,t_1\cdots t_n) \, \sum\limits_{i<j}^n g(t_i,\ t_j)  + \nonumber \\ & + 
\int\limits_0^1 v\bigl(\mu,\,F^{(1)}, \,t_1\cdots t_{n+1}\bigr) \sum_{j=1}^n g(t_j, t_{n+1})\,dt_{n+1}+ \nonumber \\
 & + \!\!\!\!\!\!
\int\limits_{\{t_{n+1}<t_{n+2}\}} \!\!\!\!\!\! v\bigl(\mu,\,F^{(2)}, \,t_1\cdots t_{n+2}\bigr) \,g(t_{n+1}, t_{n+2})\, dt_{n+1}\,dt_{n+2}.
\end{align}
\end{thm}
We refer the interested reader to \cite{MeIntRen} for a detailed treatment of the gaussian case, including the computation of all high order derivatives and proof of renormalizibility of the resulting high temperature expansion, cf. Eq. \eqref{renormone}. 

In the general ID case the equivalents of the change of measure in Eq. \eqref{GirsanovGaussian} and of the functional in Eq. \eqref{thefunctional}
are not known to us and left as open questions. Instead, we will proceed differently and replace the limit in Eq. \eqref{thedeltalimit} with the more general limit
\begin{equation}\label{thedeltalimitnew}
\frac{\partial}{\partial\delta}\Big\vert_{\delta=0}\,\,
{\bf{E}} \Bigl[ F\bigl(ze^{X(\delta)} \,M_{\mu, \varepsilon}\bigr)\bigl(ze^{X(\delta)}\bigr)^n\,e^{ \omega_{\mu,\varepsilon}(t_1)+\cdots+\omega_{\mu,\varepsilon}(t_n)}\Bigr].
\end{equation}
It is not difficult to see that the original derivation goes through intact and results in the following rule of differentiation that generalizes Theorem \ref{GMC} to the ID case
and extends Theorem \ref{theoremIDdiff} to $n>0.$ 
\begin{thm}[IDMC intermittency differentiation]\label{IDMC} 
Given $0\leq k\leq n,$ let $(p_1<\cdots <p_k)$ denote an increasing tuple of numbers from $1,\,\cdots, n$ of length $k$ 
and $\sum_{(p_1<\cdots < p_k)}$ denote the sum over all such $k-$tuples.\footnote{
$\sum_{(p_1<\cdots < p_{k})} g\bigl(\min\{t_{p_1}, t_{n+1}\}, \max\{t_{p_{k}}, t_{n+l}\}\bigr) = g(t_{n+1},\,t_{n+l})$ if $k=0.$} Fix a test function $F(x)$ and let $t_1<\cdots < t_n.$ Then,
\begin{gather}
\frac{\partial}{\partial \mu} v(\mu,\,F, \,t_1\cdots t_n)   =  \sigma^2\Bigl[v(\mu,\,F, \,t_1\cdots t_n) \, \sum\limits_{i<j}^n g(t_i,\ t_j)  + \nonumber \\  + 
\int\limits_0^1 v\bigl(\mu,\,F^{(1)}, \,t_1\cdots t_{n+1}\bigr) \sum_{j=1}^n g(t_j, t_{n+1})\,dt_{n+1}+ \nonumber \\
  + \!\!\!\!\!\!\!\!
\int\limits_{\{t_{n+1}<t_{n+2}\}} \!\!\!\!\!\!\!\! v\bigl(\mu,\,F^{(2)}, \,t_1\cdots t_{n+2}\bigr) \,g(t_{n+1}, t_{n+2})\, dt_{n+1}\,dt_{n+2}\Bigr]+ \nonumber \\
  + \sum\limits_{\substack{k+l\geq 2 \\ k\leq n}} \,\int\limits_{\mathbb{R}\setminus \{0\}} 
(e^u-1)^{k+l}\,d\mathcal{M}(u) \! \,\int\limits_{\{t_{n+1}<\cdots<t_{n+l}\}} 
\!\!\!\!\!\!\!\! 
v(\mu,\,F^{(l)}, \,t_1\cdots t_{n+l}) \times \nonumber \\ \times \sum\limits_{(p_1<\cdots < p_{k})} g\Bigl(\min\{t_{p_1}, t_{n+1}\}, \max\{t_{p_{k}}, t_{n+l}\}\Bigr) \, dt_{n+1}\cdots dt_{n+l}.\label{theruleID}
\end{gather}
\end{thm} 
A proof of Theorem \ref{IDMC} based on Eq. \eqref{thedeltalimitnew} is sketched in Section 4. 
We note that the derivation of Theorem \ref{GMC} from Eq. \eqref{thedeltalimitnew}, as opposed to
Eqs. \eqref{thedeltalimit} and \eqref{GirsanovGaussianLim}, is also new for the GMC measure.  
It is now obvious that repeated application of Eq. \eqref{theruleID} at $\mu=0$ produces the full high temperature
expansion. Its detailed analysis, especially the question of finding an analogue of the gaussian renormalizability condition
in Eq. \eqref{renormone},
  is beyond the scope of this paper and left to future research. Another interesting open
question is to identify the ID counterpart of the general class of functionals of the GMC measure in Eq. \eqref{thefunctional} 
that remain invariant under intermittency differentiation. We believe that the class of functionals that we defined in
Eq. \eqref{generfunctional} is the minimal invariant class and that it should be possible to describe a wider class by determining
the ID extension of the Girsanov transformation in Eq. \eqref{GirsanovGaussian}.

\section{Derivation of the intermittency differentiation rule}
\noindent In this section we will give derivations of our results. We note
that Theorem \ref{theoremIDdiff} is a special case of Theorems \ref{theoremIDdiffTwo} and \ref{IDMC}.
However, for the sake of clarity, we will focus in this section on the proof of Theorem \ref{theoremIDdiff} 
and only sketch the proofs of Theorems \ref{theoremIDdiffTwo} and \ref{IDMC} to minimize redundancy.

\begin{proof}[Proof of Theorem \ref{IntInv}]
The proof is based on Lemma \ref{main} and a special property of the function
$\rho_{\varepsilon}(u)$ in Eq. \eqref{rho} or, more generally, Eq. \eqref{rhogeneral}. First, define
\begin{equation}\label{rhoL}
\rho_{L,\varepsilon}(u) \triangleq \log L + \rho_{\varepsilon}(u).
\end{equation}
Then, it is easy to see from  Eq. \eqref{rho} or \eqref{rhogeneral} that $\rho_{L,\varepsilon}(u) $ satisfies the identity for $|u|<1,$
\begin{equation}\label{rhoID}
\delta+\mu\rho_{L,\varepsilon}(u) =
(\mu-\delta)\rho_{L,\varepsilon}(u)+\delta\rho_{eL,\varepsilon}(u).
\end{equation}
On the other hand, Lemma \ref{main} gives us the joint
characteristic function of $\omega_{\mu,L,\varepsilon}(t_j),$ 
$j=1\cdots n,$ $t_1<\cdots < t_n,$ in the form
\begin{equation}\label{maincharL}
{\bf E}\Biggl[\exp\Bigl(i\sum_{j=1}^n q_j \omega_{\mu,L,\varepsilon}(t_j)\Bigr)\Biggr] = \exp\Bigl(\mu\sum_{p=1}^n
\sum_{k=1}^p \alpha_{p,k} \,\rho_{L,\varepsilon}(t_p-t_k)\Bigr),
\end{equation}
where $\rho_{L,\varepsilon}(u)$ is defined in Eq. \eqref{rhoL} and
the coefficients $\alpha_{p,k}$ are the same as in Eq. \eqref{alpha}. In fact,
\begin{gather}
{\bf E}\Biggl[\exp\Bigl(i\sum_{j=1}^n q_j \omega_{\mu,L,\varepsilon}(t_j)\Bigr)\Biggr]  = {\bf E}\Bigl[\exp\Bigl(i Z_L\sum_{j=1}^n q_j \Bigr)\Bigr]
{\bf E}\Biggl[\exp\Bigl(i\sum_{j=1}^n q_j \omega_{\mu,\varepsilon}(t_j)\Bigr)\Biggr], \nonumber \\
 = \exp\Bigl(\mu \phi\bigl(\sum_{j=1}^n q_j\bigr) \log L\Bigr) \exp\Bigl(\mu\sum_{p=1}^n
\sum_{k=1}^p \alpha_{p,k} \,\rho_{\varepsilon}(t_p-t_k)\Bigr),  \nonumber \\
 = \exp\Bigl(\mu \sum_{p=1}^n\sum_{k=1}^p \alpha_{p,k} \log L\Bigr) \exp\Bigl(\mu\sum_{p=1}^n
\sum_{k=1}^p \alpha_{p,k} \,\rho_{\varepsilon}(t_p-t_k)\Bigr)
\end{gather}
by Eqs. \eqref{mainchar} and \eqref{sumID} so that Eq. \eqref{maincharL} follows from Eq. \eqref{rhoL}. 
We can now compute the joint characteristic function of the left- and right-hand sides of
Eq. \eqref{theinvariance}. Using Eqs. \eqref{sumID} and \eqref{maincharL},
\begin{equation}
{\bf E}\Biggl[\exp\Bigl(i\sum_{j=1}^n q_j \Bigl(X(\delta)+\omega_{\mu,L,\varepsilon}(t_j)\Bigr)\Biggr] = 
\exp\Bigl(\sum_{p=1}^n
\sum_{k=1}^p \alpha_{p,k} \bigl(\delta+\mu\rho_{L,\varepsilon}(t_p-t_k)\bigr)\Bigr).
\end{equation}
On the other hand, we have by independence and Eq. \eqref{maincharL},
\begin{gather}
{\bf E}\Biggl[\exp\Bigl(i\sum_{j=1}^n q_j \Bigl(\omega_{\mu-\delta,L,\varepsilon}(t_j)+\bar{\omega}_{\delta, e L, \varepsilon}(t_j)\Bigr)\Biggr]= \nonumber \\  \exp\Bigl(\sum_{p=1}^n 
\sum_{k=1}^p \alpha_{p,k} \bigl((\mu-\delta) \rho_{L,\varepsilon}(t_p-t_k)+ \delta\rho_{eL,\varepsilon}(t_p-t_k) \bigr)\Bigr),
\end{gather}
and the result follows from Eq. \eqref{rhoID}.
\end{proof}

The derivation of the intermittency differentiation rule requires three auxiliary lemmas. Recall the definition of the L\'evy process in Eq. \eqref{Xdelta}
that is associated with the ID distribution specified by its L\'evy-Khinchine representation in Eq. \eqref{phi}.
\begin{lemma}[Kolmogorov equation]\label{KolmogorovID}
Given a test function $v(z),$ 
\begin{align}
\frac{\partial}{\partial\delta}\Big\vert_{\delta=0}\,\,
{\bf E} \left[v\bigl(z e^{X(\delta)}\bigr)\right] = & \frac{\sigma^2}{2}z^2\frac{d^2}{dz^2} v(z) + \nonumber \\ + &
\int_{\mathbb{R}\setminus \{0\}} \Bigl[v(ze^u)-v(z)-z \frac{d}{dz}v(z) (e^u-1)\Bigr] d\mathcal{M}(u).
\end{align} 
\end{lemma}
\begin{proof}
This is a simple corollary of the backward Kolmogorov equation for the process $X(\delta),$ cf. \cite{Bertoin}, Section I.2.
It is easy to see by following the Fourier-integral type of argument given in \cite{Bertoin} that the backward Kolmogorov operator
associated with $X(\delta),$ 
\begin{equation}
(\mathcal{L}f)(x) \triangleq \frac{\partial }{\partial\delta}\Big\vert_{\delta=0} {\bf
E}\Bigl[f\bigl(X(\delta)+x\bigr)\Bigr],
\end{equation}
is
\begin{align}
(\mathcal{L}f)(x) = & -\frac{\sigma^2}{2}\frac{d}{dx}f(x) + \frac{\sigma^2}{2} \frac{d^2}{dx^2}f(x) + \nonumber \\ & +
\int_{\mathbb{R}\setminus \{0\}} \Bigl[ f(x+u)-f(x)-\frac{d}{dx}f(x) (e^u-1)\Bigr] d\mathcal{M}(u).
\end{align} 
It remains to apply this formula to $f(x)=v(ze^x)$ at $x=0.$ 
\end{proof}
Now, recall definitions of $d(m)$ in Eq. \eqref{d}, of the process $\omega_{\mu,L,\varepsilon}(s)$ in Eq. \eqref{omegaLprocess},
 and of $\rho_{L,\varepsilon}(u)$ in Eq. \eqref{rhoL}.
\begin{lemma}[Combinatorial property]\label{BLemma}
Let $\mathfrak{f}(\delta,s)$ be an arbitrary continuous function
that vanishes as $\delta\rightarrow 0.$ Let
$\mathcal{B}(s)\triangleq
e^{\mathfrak{f}(\delta,s)+\omega_{\delta,L,\varepsilon}(s)}-1.$
Then, given any distinct $0<s_1<\cdots< s_n<1,$ $n\geq 2,$ as
$\delta\rightarrow 0,$
\begin{align}
{\bf E} \bigl[\mathcal{B}(s_1)\mathcal{B}(s_2)\bigr] = &
\bigl(e^{\mathfrak{f}(\delta,s_1)}-1\bigr)\bigl(e^{\mathfrak{f}(\delta,s_2)}-1\bigr)+ \nonumber \\ & +
\delta \rho_{L,\varepsilon}(s_2-s_1) \Bigl(\sigma^2 +  \!\!\!\!\!\!\int\limits_{\mathbb{R}\setminus \{0\}} \!\! (e^u-1)^2 d\mathcal{M}(u)\Bigr)+
o(\delta), \label{B2ID1}\\
{\bf E} \bigl[\mathcal{B}(s_1)\cdots\mathcal{B}(s_n)\bigr] = &
\bigl(e^{\mathfrak{f}(\delta,s_1)}-1\bigr)\cdots\bigl(e^{\mathfrak{f}(\delta,s_n)}-1\bigr)+ \nonumber \\ & +
\delta \rho_{L,\varepsilon}(s_n-s_1) \int\limits_{\mathbb{R}\setminus \{0\}}  (e^u-1)^n d\mathcal{M}(u)+
o(\delta),\label{B2ID2}
\end{align}
if $n>2.$
\end{lemma}
This lemma generalizes the corresponding result for the gaussian multiplicative chaos measure that we first
noted in \cite{Me2}.
\begin{proof}
We need the following estimate first,
\begin{align}
{\bf E} \bigl[\mathcal{B}(s_1)\cdots\mathcal{B}(s_n)\bigr] = &
\bigl(e^{\mathfrak{f}(\delta,s_1)}-1\bigr)\cdots\bigl(e^{\mathfrak{f}(\delta,s_n)}-1\bigr)+\nonumber \\ & +\delta
\rho_{L,\varepsilon}(s_n-s_1)\sum\limits_{k=0}^{n-2} (-1)^k\,\binom{n-2}{k}\,d_{n-k-1}  + 
o(\delta).  \label{B2ID}
\end{align}
The proof is essentially based on the identity that follows from Eq. \eqref{maincharL},
\begin{equation}\label{interm}
\exp\Bigl(\mu\sum_{k<p}^n d(p-k) \, \rho_{L,\varepsilon}(s_p- s_k) \Bigr) =  {\bf E}\Bigl[
e^{ \omega_{\mu,L,\varepsilon}(s_1)+\cdots+\omega_{\mu,L,\varepsilon}(s_n)}\Bigr] 
\end{equation}
for any $0<s_1<\cdots<s_n<1.$ 
One then multiplies out the terms on the left-hand side of Eq. \eqref{B2ID}, applies this identity to each resulting
term, and differentiates the result with respect to $\delta.$ To carry out this calculation in detail, we need to introduce the following notation.
Fix $n\geq 2$ and let $(p_k<\cdots < p_k)$ denote a $k-$tuple of numbers from $\{1,\,\cdots, n\}.$ Then, we have the obvious identity
that follows from Eq. \eqref{interm} and the vanishing of $\mathfrak{f}(\delta, s)$ as $\delta\rightarrow 0,$ 
\begin{align}
{\bf E} \bigl[\mathcal{B}(s_1)\cdots\mathcal{B}(s_n)\bigr] = & \prod\limits_{i=1}^n\bigl(e^{\mathfrak{f}(\delta,s_i)}-1\bigr)+\nonumber \\   + & \delta
\sum\limits_{k=2}^{n} (-1)^{n-k} \!\!\!\!\sum\limits_{(p_1<\cdots< p_k)} \sum\limits_{i<j}^k d(j-i) \rho_{L,\varepsilon}(s_{p_j}-s_{p_i}) +
o(\delta).
\end{align}
Now, consider the coefficient of $\rho_{L,\varepsilon}(s_n-s_1)$ first. This means $p_1=1$ and $p_k=n$ so that there
are 
\begin{equation}
\binom{n-2}{k-2}
\end{equation}
such tuples. Hence, the coefficient of $\rho_{L,\varepsilon}(s_n-s_1)$ is
\begin{equation}
\delta
\sum\limits_{k=2}^{n} (-1)^{n-k}  \binom{n-2}{k-2} d(k-1),
\end{equation}
which coincides with the expression on the right-hand side of Eq. \eqref{B2ID} by a change of summation index.  Next, consider the coefficient of
$d(j-i) \,\rho_{L,\varepsilon}(s_{p_j}-s_{p_i})$ in general. Let $p_i=a,$ $p_j=b$ and $l=j-i$ be fixed. Given, $a<b$ and $l,$ we wish to show that
the number of $k-$tuples that have the property that they contain $a$ and $b,$ \emph{i.e.} $a=p_i$ and $b=p_j$ for some $i<j,$ and $j-i=l,$ is
\begin{equation}\label{interm2}
\binom{n-1-(b-a)}{k-1-l}\binom{b-a-1}{l-1}.
\end{equation}
For example, let $n=7,$ $k=4,$ $a=2,$ $b=5,$ and $l=2.$ The formula says that there are six such tuples. 
In fact, they are: (2, 3, 5, 6), (2, 3, 5, 7), (2, 4, 5, 6), (2, 4, 5, 7), (1, 2, 3, 5), (1, 2, 4, 5). 
To prove the formula, let $i$ be the location of $a$ so the location of $b$ is then necessarily $j=i+l.$
The sought number of tuples is then
\begin{equation}
\sum\limits_{i} \binom{a-1}{i-1}\binom{b-a-1}{l-1}\binom{n-b}{k-j}.
\end{equation}
This formula simply gives us numbers of choices for the elements of the tuple preceding $a,$ in between $a$ and $b,$ and following $b.$
The formula in Eq. \eqref{interm2} now follows by the Vandermonde convolution, cf. Eq. (3.1) in \cite{G}. 
It follows that the coefficient of $d(l) \rho_{L,\varepsilon}(s_{b}-s_{a})$ is
\begin{equation}
\delta
\sum\limits_{k=2}^{n} (-1)^{n-k}  \binom{n-1-(b-a)}{k-1-l}\binom{b-a-1}{l-1}.
\end{equation}
It remains to observe that this sum is identically zero provided $n-1>b-a,$ which is a corollary of the classical binomial identity
\begin{equation}
\sum\limits_{k\geq 0} (-1)^k \binom{x}{k}=0, \,\, x>0,
\end{equation}
cf. Eq. (1.2) in \cite{G}. The case of $n-1=b-a$ was already treated above, hence Eq. \eqref{B2ID} is verified.
Finally, 
the alternating sum on the right-hand side of Eq. \eqref{B2ID} can be simplified 
using the definition of $d(m)$ in Eq. \eqref{d} resulting in Eqs. \eqref{B2ID1} and \eqref{B2ID2}.
\end{proof}
\begin{lemma}[Differentiation]\label{lemmadiffID}
Let $F(x)$ be a smooth test function. Let
\begin{equation} u_\varepsilon(z, \,\mu, \,F) \triangleq
F\Bigl(z\int_0^1 e^{\omega_{\mu, L, \varepsilon}(s)}\,ds\Bigr).
\end{equation}
Then, there holds the following identity
\begin{gather}
\frac{\partial}{\partial\mu} u_\varepsilon(z, \,\mu, \,F) =
-\lim\limits_{\delta\rightarrow 0} \Bigl[\frac{1}{\delta} \sum_{k=1}^\infty
\frac{ u_\varepsilon(z,
\,\mu,\, F^{(k)})}{k!}\Bigl(\!z\!\!\int\limits_0^1 e^{\omega_{\mu, L, \varepsilon}(s)}
\bigl(e^{\mathcal{A}_\varepsilon(s)}-1\bigr)\,ds\Bigr)^k\Bigr],
\end{gather}
where
\begin{equation}\label{intermeq5ID}
\mathcal{A}_\varepsilon(s) \triangleq
\omega_{\mu-\delta,L,\varepsilon}(s)-\omega_{\mu,L,\varepsilon}(s).
\end{equation}
\end{lemma}
\begin{proof}
The result follows from writing
\begin{align}
\int_0^1 e^{\omega_{\mu-\delta, L, \varepsilon}(s)}\,ds = & 
\int_0^1 e^{\omega_{\mu, L, \varepsilon}(s)}\,ds   +
\int_0^1 e^{\omega_{\mu, L, \varepsilon}(s)}\,\bigl( e^{\mathcal{A}_\varepsilon(s)}-1 \bigr)ds,
\end{align}
and Taylor expanding in the ``small" parameter \[\int_0^1 e^{\omega_{\mu, L, \varepsilon}(s)}
\bigl(e^{\mathcal{A}_\varepsilon(s)}-1\bigr) ds\] that vanishes as
$\delta\rightarrow 0.$ 
\end{proof}
We can now give a derivation of Theorem \ref{theoremIDdiff}.
\begin{proof} 
The idea of the derivation is to consider a stochastic flow and derive the corresponding
Feynman-Kac equation regarding intermittency as time.
Let 
\begin{equation}
u_\varepsilon(z,\,\mu,\,F) \triangleq F\Bigl(z\int_0^1
e^{\omega_{\mu,1,\varepsilon}(s)}ds\Bigr)
\end{equation}
and let
$v_\varepsilon(z,\,\mu,\,F)$ be its expectation,
\begin{equation}
v_\varepsilon(z,\,\mu,\,F) \triangleq {\bf
E}\bigl[u_\varepsilon(z,\,\mu,\,F)\bigr].
\end{equation}
The starting point is the limit
\begin{equation}\label{thelim}
A \triangleq \frac{\partial}{\partial\delta}\Big\vert_{\delta=0}\,\,
{\bf{E}^*} \left[v_\varepsilon \left(z
e^{X(\delta)},\,\mu,\,F\right)\right],
\end{equation}
where $X(\delta)$ is defined by Eq. \eqref{Xdelta} and is independent of $\omega_{\mu,1,\varepsilon}(s),$ 
and the star is used to distinguish the expectation with respect to $X(\delta)$ from that with respect to
$\omega_{\mu,1,\varepsilon}(s).$
By Lemma \ref{KolmogorovID}, we have
\begin{align}
A = & \frac{\sigma^2}{2}z^2\frac{\partial^2}{\partial z^2} v_\varepsilon(z,\,\mu,\,F) + \nonumber \\ &  +
\int_{\mathbb{R}\setminus \{0\}} \Bigl[v_\varepsilon(ze^u,\,\mu,\,F) -v_\varepsilon(z,\,\mu,\,F)-z \frac{\partial}{\partial z}v_\varepsilon(z,\,\mu,\,F) (e^u-1)\Bigr] d\mathcal{M}(u). \label{firsteq}
\end{align}

On the other hand, this limit can be computed in a different way.
By Theorem \ref{IntInv}, there holds the following equality in law 
\begin{equation}
e^{X(\delta)} \int_0^1 e^{\omega_{\mu, 1,
\varepsilon}(s)} ds =\int_0^1 e^{\omega_{\mu-\delta,1,\varepsilon}(s)+{\bar
\omega}_{\delta,e,\varepsilon}(s)} ds.
\end{equation}
Thus, to compute the limit in Eq. \eqref{thelim}, we need to expand
\begin{equation}\label{intermeq4}
{\bf E}^*\left[{\bf E}\Bigl[F\Bigl(z \int_0^1
e^{\omega_{\mu-\delta,1,\varepsilon}(s)+{\bar
\omega}_{\delta,e,\varepsilon}(s)} ds\Bigr)\Bigr]\right] -
v_\varepsilon(z,\,\mu,\,F)
\end{equation}
in $\delta$ up to $o(\delta)$ terms. The star now indicates the expectation with respect to
${\bar
\omega}_{\delta,e,\varepsilon}(s),$ which is independent of $\omega_{\mu-\delta,1,\varepsilon}(s)$ by construction. 
Let $\mathcal{A}_\varepsilon(s)
\triangleq
\omega_{\mu-\delta,1,\varepsilon}(s)-\omega_{\mu,1,\varepsilon}(s)$
as in Eq. \eqref{intermeq5ID} with $L=1$ and
\begin{equation}
\bar{\mathcal{A}}_\varepsilon(s)\triangleq {\bar
\omega}_{\delta,e,\varepsilon}(s).
\end{equation}
While we do not know how to expand either
$\mathcal{A}_\varepsilon(s)$ or $\bar{\mathcal{A}}_\varepsilon(s)$
in $\delta,$ they both clearly vanish as $\delta\rightarrow 0.$ It
follows that the expression in Eq. \eqref{intermeq4} can be expanded
in the ``small'' quantity
\begin{equation}\label{Cquantity}
\mathcal{C}\triangleq \int_0^1 e^{\omega_{\mu,1,\varepsilon}(s)}
\bigl(e^{\mathcal{A}_\varepsilon(s)+\bar{\mathcal{A}}_\varepsilon(s)}-1\bigr)\,ds.  
\end{equation}
\begin{align}
{\bf E}^*\!\!\left[{\bf E}\Bigl[F\Bigl(z \!\!\int_0^1
e^{\omega_{\mu-\delta,1,\varepsilon}(s)+{\bar
\omega}_{\delta,e,\varepsilon}(s)} ds\Bigr)\Bigr]\right] 
=  & {\bf E}^*\!\!\left[{\bf E}\Bigl[F\Bigl(z \!\!\int_0^1
e^{\omega_{\mu,1,\varepsilon}(s)}  ds  + z \mathcal{C}\Bigr)\Bigr]\right], \nonumber \\
= & {\bf E}^*\!\!\left[{\bf E}\Bigl[\sum_{k=0}^\infty
\frac{z^k}{k!} u_\varepsilon(z,\,\mu,\,F^{(k)})\,\mathcal{C}^k\Bigr]\right]. \label{intermeq3}
\end{align}
The advantage of this representation is that the only
$\bar{\omega}_\varepsilon$ dependence is in
$\bar{\mathcal{A}}_\varepsilon(s).$ This allows us to compute the
$\bf{E}^*$ expectation in Eq. \eqref{intermeq4}. Indeed, Eq. \eqref{intermeq4} entails two
expectations: the $\bf{E}$ with respect to $\omega_\varepsilon$
process inherited from the definition of $v_\varepsilon(z,\mu,F)$
and the $\bf{E}^*$ expectation with respect to
$\bar{\omega}_\varepsilon$ process. Interchanging their order, it
follows from Eq. \eqref{intermeq3} that computing the $\bf{E}^*$ expectation is
now reduced to computing ${\bf E}^*[\mathcal{C}^k].$ As
$\mathcal{A}_\varepsilon(s)$ and $\bar{\mathcal{A}}_\varepsilon(s)$
are independent processes, it follows from Lemma \ref{BLemma} applied to
$\mathcal{B}(s)=\exp\bigl(\mathcal{A}_\varepsilon(s)+\bar{\mathcal{A}}_\varepsilon(s)\bigr)-1$
with $\mathfrak{f}(\delta, s)\triangleq \mathcal{A}_\varepsilon(s)$
that the $\bf{E}^*$ expectation equals
\begin{align}
{\bf E}^*[\mathcal{B}(s_1)\mathcal{B}(s_2)]= &
\bigl(e^{\mathcal{A}_\varepsilon(s_1)}-1\bigr)\bigl(e^{\mathcal{A}_\varepsilon(s_2)}-1\bigr)+ \nonumber \\ & +
\delta \,\rho_{e,\varepsilon}(s_2-s_1) \Bigl(\sigma^2 +  \!\!\!\!\!\!\int\limits_{\mathbb{R}\setminus \{0\}} \!\! (e^u-1)^2 d\mathcal{M}(u)\Bigr)+o(\delta), \\
{\bf E}^*[\mathcal{B}(s_1)\cdots\mathcal{B}(s_k)]= & 
\bigl(e^{\mathcal{A}_\varepsilon(s_1)}-1\bigr)\cdots\bigl(e^{\mathcal{A}_\varepsilon(s_k)}-1\bigr)+\nonumber \\&  +
\delta \rho_{e,\varepsilon}(s_k-s_1) \int\limits_{\mathbb{R}\setminus \{0\}}  (e^u-1)^k d\mathcal{M}(u) +
o(\delta),
\end{align}
if $k>2.$
Collecting what we have shown so far, we obtain
\begin{gather}
A=
\lim\limits_{\delta\rightarrow 0}\frac{1}{\delta}\sum_{k=1}^\infty
\frac{z^k}{k!}{\bf E}\Bigl[u_\varepsilon(z, \,\mu,\,F^{(k)})
\Bigl(\int_{[0,1]} e^{\omega_{\mu,1,\varepsilon}(s)}
\bigl(e^{\mathcal{A}_\varepsilon(s)}-1\bigr)\,ds\Bigr)^k\Bigr]+ \nonumber\\
+\sigma^2 z^2 \int\limits_{\{s_1<s_2\}} {\bf E}\Bigl[u_\varepsilon(z, \,\mu,\,F^{(2)}) e^{\omega_{\mu,1,\varepsilon}(s_1)+\omega_{\mu,1,\varepsilon}(s_2)}\Bigr]
\,\rho_{e,\varepsilon}(s_2-s_1)\,ds^{(2)}+ \nonumber \\
+\sum\limits_{k=2}^\infty z^k \int\limits_{\mathbb{R}\setminus \{0\}} 
(e^u-1)^k\,d\mathcal{M}(u) \int\limits_{\{s_1<\cdots<s_k\}}
{\bf E}\Bigl[u_\varepsilon(z, \,\mu,\,F^{(k)})\times \nonumber \\ \times
e^{ \omega_{\mu,1,\varepsilon}(s_1)+\cdots+\omega_{\mu,1,\varepsilon}(s_k)}\Bigr] 
\rho_{e,\varepsilon}(s_k-s_1)\,ds^{(k)}
.\label{intermeq1}
\end{gather}
By Lemma \ref{lemmadiffID}, the $\delta\rightarrow 0$ limit that is involved in Eq. 
\eqref{intermeq1} equals
\begin{equation}\label{intermeq2}
-\frac{\partial}{\partial\mu} v_\varepsilon(z, \,\mu,\,F).
\end{equation}
Observing that
$\rho_{e,\varepsilon}(s_2-s_1)=1+\rho_{\varepsilon}(s_2-s_1)$ and
\begin{gather}
\int\limits_{\{s_1<s_2\}} {\bf E}\Bigl[u_\varepsilon(z, \,\mu,\,F^{(2)}) e^{\omega_{\mu,1,\varepsilon}(s_1)+\omega_{\mu,1,\varepsilon}(s_2)}\Bigr] \,ds^{(2)} = \frac{1}{2}\frac{\partial^2}{\partial z^2}
v_\varepsilon(z,\mu,\,F), \\
 \sum\limits_{k=2}^\infty z^k \!\!\!\!\!\int\limits_{\mathbb{R}\setminus \{0\}} \!\!
(e^u-1)^k\,d\mathcal{M}(u) \!\!\!\!\!\!\!\!\!\!\!\int\limits_{\{s_1<\cdots<s_k\}} \!\!\!\!\!\!\!\!\!\!\!
{\bf E}\Bigl[u_\varepsilon(z, \,\mu,\,F^{(k)})
e^{ \omega_{\mu,1,\varepsilon}(s_1)+\cdots+\omega_{\mu,1,\varepsilon}(s_k)}\Bigr] 
\,ds^{(k)}= \nonumber \\
=\int_{\mathbb{R}\setminus \{0\}} \Bigl[v_\varepsilon(ze^u,\,\mu,\,F) -v_\varepsilon(z,\,\mu,\,F)-z \frac{\partial}{\partial z}v_\varepsilon(z,\,\mu,\,F) (e^u-1)\Bigr] d\mathcal{M}(u),
\end{gather}
and substituting 
these equations into Eq. \eqref{intermeq1}, we obtain
\begin{gather}
A= -\frac{\partial}{\partial\mu} v_\varepsilon(z, \,\mu,\,F) +
\frac{\sigma^2 z^2}{2}\frac{\partial^2}{\partial z^2}
v_\varepsilon(z,\mu,F)+\nonumber\\
+ \int_{\mathbb{R}\setminus \{0\}} \Bigl[v_\varepsilon(ze^u,\,\mu,\,F) -v_\varepsilon(z,\,\mu,\,F)-z \frac{\partial}{\partial z}v_\varepsilon(z,\,\mu,\,F) (e^u-1)\Bigr] d\mathcal{M}(u)+\nonumber \\
+ \sigma^2 z^2 \int\limits_{\{s_1<s_2\}} {\bf E}\Bigl[u_\varepsilon(z, \,\mu,\,F^{(2)})
e^{\omega_{\mu,1,\varepsilon}(s_1)+\omega_{\mu,1,\varepsilon}(s_2)}\Bigr]
\,\rho_{\varepsilon}(s_2-s_1)ds^{(2)}+ \nonumber \\ 
+\sum\limits_{k=2}^\infty z^k \int\limits_{\mathbb{R}\setminus \{0\}} 
(e^u-1)^k\,d\mathcal{M}(u) \int\limits_{\{s_1<\cdots<s_k\}}
{\bf E}\Bigl[u_\varepsilon(z, \,\mu,\,F^{(k)})\times \nonumber \\ \times
e^{ \omega_{\mu,1,\varepsilon}(s_1)+\cdots+\omega_{\mu,1,\varepsilon}(s_k)}\Bigr] 
\rho_{\varepsilon}(s_k-s_1)\,ds^{(k)}.\label{lasteq}
\end{gather}
Finally, upon 
comparing this expression for $A$ with that in Eq. \eqref{firsteq},
and then letting $z=1$ and $\varepsilon\rightarrow 0,$ we arrive at Eq. \eqref{theIDrule}. 
\end{proof}

\begin{proof}[Proof of Theorem \ref{theoremIDdiffTwo}]
The proof is very similar to that of Theorem \ref{theoremIDdiff} so it is sufficient to point out
that instead of the limit in Eq. \eqref{thelim} one needs to evaluate the more general limit
\begin{equation}
A \triangleq \frac{\partial}{\partial\delta}\Big\vert_{\delta=0}\,\,
{\bf{E}} \Bigl[F_1\Bigl(ze^{X(\delta)}\int_{I_1} e^{\omega_{\mu,1,\varepsilon}(s)}ds\Bigr) \,
F_2\Bigl(ze^{X(\delta)}\int_{I_2} e^{\omega_{\mu,1,\varepsilon}(s)}ds\Bigr) 
\Bigr],
\end{equation}
The remaining details are essentially the same and will be omitted. 
\end{proof}

\begin{proof}[Proof of Corollary \ref{covID}]
It is obvious that the covariance in Eq. \eqref{covstrucID} is linear in $\mu$ in the limit $\tau\rightarrow 0.$
The slope can be computed by Corollary \ref{TwoIntervalsZeroInter}. 
Let 
\begin{equation}
F_1(x)=F_2(x)=\log x.
\end{equation}
The intervals $I_1$ and $I_2$ are
\begin{equation}
I_1=[0,\,\tau],\;\,I_2=[t,\,t+\tau],
\end{equation}
and the expansion is around the point $x=|I_1|=|I_2|=\tau$ 
so that the necessary derivative is
\begin{equation}
F^{(k)}(\tau) = (-1)^{k-1} (k-1)! \tau^{-k}.
\end{equation}
Using the standard formula for the covariance of two random variables
\begin{equation}
{\bf Cov}\bigl(A, \,B\bigr) = {\bf E} \bigl[A\,B\bigr] - {\bf E} \bigl[A\bigr] {\bf E} \bigl[B\bigr],
\end{equation}
and applying Corollary \ref{TwoIntervalsZeroInter} 
it is not difficult to see that $\log\tau$ terms cancel
out and the remaining terms are
\begin{gather}
\frac{\partial}{\partial \mu} \Big\vert_{\mu=0} {\bf Cov}\Bigl(\log\int\limits_t^{t+\tau}M_\mu(ds), \,\log\int\limits_0^{\tau} M_\mu(ds)\Bigr) =  
\frac{\sigma^2 }{\tau^2}\!\!\int\limits_{\{s_1\in I_1,\, s_2\in I_2\}}\!\!\!
g(s_1,\,s_2)\,ds^{(2)} + \nonumber \\ + 
\sum\limits_{k,l\geq 1}^\infty  (-1)^{k+l} (k-1)! (l-1)! \tau^{-(k+l)}
 \int\limits_{\mathbb{R}\setminus \{0\}} 
(e^u-1)^{k+l}\,d\mathcal{M}(u) \times \nonumber \\   \times \int\limits_{\substack{\{s_1<\cdots<s_k\}\subset I_1 \\ \{s_{k+1}<\cdots<s_{k+l}\}\subset I_2 }} \!\!\!\!\!\!\!\!\!\!\!\!
g(s_1,\,s_{k+l})\,ds^{(k+l)}. 
\end{gather}
Finally, in the limit $\tau\rightarrow 0$ the integrals can be obviously approximated by 
\begin{equation}
\int\limits_{\substack{\{s_1<\cdots<s_k\}\subset I_1 \\ \{s_{k+1}<\cdots<s_{k+l}\}\subset I_2 }} \!\!\!\!\!\!\!\!\!\!\!\!
g(s_1,\,s_{k+l})\,ds^{(k+l)} = \frac{\tau^{k+l}}{k! l!}\, g(t) +O (\tau),
\end{equation}
and the result follows from the power series expansion of the logarithm function.
\end{proof}

\begin{proof} [Proof of Theorem \ref{intidentity}]
We first need to simplify the sum in Eq. \eqref{derivsinglemomformula}. The starting point is the identity
\begin{align}
\sum\limits_{k<p}^n d(p-k)\,g(t_p, t_k) = & \sigma^2\,\sum\limits_{k<p}^n g(t_p, t_k)  + \sum\limits_{k=2}^{n} \,\int\limits_{\mathbb{R}\setminus \{0\}} (e^u-1)^k\,d\mathcal{M}(u) \times \nonumber \\ & \times \Bigl[\sum\limits_{\substack{i<j\\j-i\geq k-1}}^n 
\binom{j-i-1}{k-2} g(t_i, t_j)\Bigr]. \label{dsumeq}
\end{align}
Its proof is a simple application of Eq. \eqref{d}, which implies
\begin{align}
\sum\limits_{k<p}^n d(p-k)\,g(t_p, t_k) = & \sigma^2\,\sum\limits_{k<p}^n g(t_p, t_k)  + \nonumber \\ & +
\int\limits_{\mathbb{R}\setminus \{0\}} (e^u-1)^2\,d\mathcal{M}(u)
\Bigl[\sum\limits_{k<p}^n e^{(p-k-1)u} \,g(t_k, t_p)\Bigr].
\end{align}
Now, using the identity 
\begin{equation}
e^{pu} = \sum\limits_{s=0}^p (e^u-1)^s \binom{p}{s},
\end{equation}
and changing the order of summation, we obtain
\begin{align}
\sum\limits_{k<p}^n d(p-k)\,g(t_p, t_k) = & \sigma^2\,\sum\limits_{k<p}^n g(t_p, t_k)  + 
\int\limits_{\mathbb{R}\setminus \{0\}} d\mathcal{M}(u) \times\ \nonumber \\  & \times
\Bigl[\sum\limits_{s=0}^{n-2} (e^{u}-1)^{s+2} \sum\limits_{\substack{k<p \\ p-k-1\geq s}}^n \binom{p-k-1}{s} \,g(t_k, t_p)\Bigr],
\end{align}
which is equivalent to Eq. \eqref{dsumeq}. 

We note next that the product in Eq. \eqref{derivsinglemomformula} satisfies
\begin{equation}
\prod\limits_{k<p}^n r(t_p-t_k)^{-\mu\,d(p-k)} = \lim\limits_{\varepsilon\rightarrow
0} {\bf E}\Bigl[
e^{ \omega_{\mu,\varepsilon}(t_1)+\cdots+\omega_{\mu,\varepsilon}(t_n)}\Bigr] 
\end{equation}
for any $0<t_1<\cdots<t_n<1,$ cf. Eq. \eqref{momrhoidentity}.
Then, upon substituting this equation into Eq. \eqref{derivsinglemomformula}, we see that to establish the equivalence of
Eqs. \eqref{derivsinglemomformula} and \eqref{theIDruleMom} it is sufficient to verify the
identity in Eq. \eqref{intidentityeq}. 

Before we give a formal proof of Eq. \eqref{intidentityeq}, we will treat the special case of $n-k=1$ to illustrate the main idea. 
By breaking up the integration region of the $ds$ integral into three subregions, we can write
\begin{gather}
\int_0^1 e^{\omega(s)} ds \!\!\!\!\!\!\!
\int\limits_{\{s_1<\cdots<s_k\}} \!\!\!
e^{ \omega(s_1)+\cdots+\omega(s_k)}
g(s_1,\,s_k)\,ds^{(k)}= \nonumber \\ =\int\limits_{\{s<s_1<\cdots<s_k\}} \!\!\!
e^{ \omega(s)+\omega(s_1)+\cdots+\omega(s_k)}
g(s_1,\,s_k)\,ds^{(k+1)} + \nonumber \\
+  \sum\limits_{i=1}^{k-1} \int\limits_{\{s_1<\cdots<s_i<s<s_{i+1}<\cdots<s_k\}} \!\!\!
e^{ \omega(s)+\omega(s_1)+\cdots+\omega(s_k)}
g(s_1,\,s_k)\,ds^{(k+1)}+ \nonumber \\
+ \int\limits_{\{s_1<\cdots<s_k<s\}} \!\!\!
e^{ \omega(s)+\omega(s_1)+\cdots+\omega(s_k)}
g(s_1,\,s_k)\,ds^{(k+1)}.
\end{gather}
We now relabel the variables of integration, resulting in the identity
\begin{gather}
\int_0^1 e^{\omega(s)} ds \!\!\!\!\!\!\!
\int\limits_{\{s_1<\cdots<s_k\}} \!\!\!
e^{ \omega(s_1)+\cdots+\omega(s_k)}
g(s_1,\,s_k)\,ds^{(k)}= \nonumber \\ =\int\limits_{\{s_1<\cdots<s_{k+1}\}} \!\!\!
e^{ \omega(s_1)+\cdots+\omega(s_{k+1})} \Bigl[
g(s_2,\,s_{k+1})  +\nonumber \\ +  (k-1)  g(s_1,\,s_{k+1}) + g(s_1,\,s_{k})\Bigr]ds^{(k+1)}, 
\end{gather}
which is the same as the expression on the right-hand side of Eq. \eqref{intidentityeq}. 
To prove Eq. \eqref{intidentityeq} in general, it is clear that one can apply the above procedure iteratively
resulting in an identity of the form
\begin{gather}
\frac{1}{(n-k)!}  \Bigl(\int_0^1 e^{\omega(s)} ds\Bigr)^{n-k}  \!\!\!\!\!\!\!
\int\limits_{\{s_1<\cdots<s_k\}} \!\!\!
e^{ \omega(s_1)+\cdots+\omega(s_k)}
g(s_1,\,s_k)\,ds^{(k)}= \nonumber \\ =\int\limits_{\{s_1<\cdots<s_n\}}  
e^{ \omega(s_1)+\cdots+\omega(s_n)}  \Bigl[\sum\limits_{i<j}^n 
C_{ijk} \,g(s_i, s_j)\Bigr] ds^{(n)}
\end{gather}
with some coefficients $C_{ijk}$ to be determined. Now, the left-hand side can be obviously written as
\begin{gather}
\frac{1}{(n-k)!}  \Bigl(\int_0^1 e^{\omega(s)} ds\Bigr)^{n-k}  \!\!\!\!\!\!\!
\int\limits_{\{s_1<\cdots<s_k\}} \!\!\!
e^{ \omega(s_1)+\cdots+\omega(s_k)}
g(s_1,\,s_k)\,ds^{(k)}  = \nonumber \\
\int\limits_{\{t_1<\cdots<t_{n-k}\}}  e^{ \omega(t_1)+\cdots+\omega(t_{n-k})} dt^{(n-k)}\!\!\!\!\!\!\!
\int\limits_{\{s_1<\cdots<s_k\}} \!\!\!
e^{ \omega(s_1)+\cdots+\omega(s_k)}
g(s_1,\,s_k)\,ds^{(k)}.
\end{gather}
Then, following the logic of the calculation in the case of $n-k=1$ above, the coefficient $C_{ijk}$ equals
the number of ways of inserting the $t's$ into $s_1<\cdots<s_k$ so that $s_1$ ends up in the position $i$ and
$s_k$ in the position $j$ after the insertion. This means that there should be $j-i-1$ variables in between $s_1$ 
and $s_k,$ of which there are $k-2$ $s_i$s and $j-i-1-(k-2)$ $t_i$s. Thus, $C_{ijk}$ equals the number of
ways of choosing $k-2$ locations out of $j-i-1$ available positions, \emph{i.e.}
\begin{equation}
C_{ijk} = \binom{j-i-1}{k-2}.
\end{equation}
This completes the proof. 
\end{proof}

\begin{proof}[Proof of Theorem \ref{IDMC}]
The details of the calculation are quite similar to those of the proof of Theorem \ref{theoremIDdiff} so we will only highlight
the key steps. By Theorem \ref{IntInv} we have the identity in law,
\begin{align}
F\Bigl(z e^{X(\delta)}\!\!\! \int_0^1 e^{\omega_{\mu, 1,
\varepsilon}(s)} ds\Bigr) \prod\limits_{i=1}^n e^{ \omega_{\mu, 1, \varepsilon}(t_i) + X(\delta)} &
=F\Bigl(z\!\!\int_0^1 e^{\omega_{\mu-\delta,1,\varepsilon}(s)+{\bar
\omega}_{\delta,e,\varepsilon}(s)} ds\Bigr) \times \nonumber \\ & \times \prod\limits_{i=1}^n e^{ \omega_{\mu-\delta,1,\varepsilon}(t_i)+{\bar
\omega}_{\delta,e,\varepsilon}(t_i) }.\label{zidentity}
\end{align}
Denote
\begin{equation}
M_{\mu, \varepsilon} = \int_0^1 e^{\omega_{\mu, 1,
\varepsilon}(s)} ds
\end{equation}
as before. Hence, we can compute the limit 
\begin{equation}
A\triangleq \frac{\partial}{\partial\delta}\Big\vert_{\delta=0}\,\,
{\bf{E}} \Bigl[ F\bigl(ze^{X(\delta)} \,M_{\mu, \varepsilon}\bigr)\bigl(ze^{X(\delta)}\bigr)^n\,e^{ \omega_{\mu,\varepsilon}(t_1)+\cdots+\omega_{\mu,\varepsilon}(t_n)}\Bigr]
\end{equation}
in two ways: by Lemma \ref{KolmogorovID} and by expanding the right-hand side in $\delta$ using Eq. \eqref{zidentity}.
Let $\mathcal{A}_\varepsilon(s),$ $\bar{\mathcal{A}}_\varepsilon(s),$ and $\mathcal{C}$ be as in the proof of Theorem \ref{theoremIDdiff}.
We can write
\begin{gather}
F\Bigl(z\!\!\int_0^1 e^{\omega_{\mu-\delta,1,\varepsilon}(s)+{\bar
\omega}_{\delta,e,\varepsilon}(s)} ds\Bigr) = \sum\limits_{l=0}^\infty \frac{z^l}{l!} F^{(l)}(z  M_{\mu, \varepsilon}) \,\mathcal{C}^l, \\
\prod\limits_{i=1}^n e^{ \omega_{\mu-\delta,1,\varepsilon}(t_i)+{\bar
\omega}_{\delta,e,\varepsilon}(t_i) } = \prod\limits_{i=1}^n e^{ \omega_{\mu,1,\varepsilon}(t_i) } 
\prod\limits_{i=1}^n \Bigl(1+ \bigl(e^{ \mathcal{A}_\varepsilon(t_i) + \bar{\mathcal{A}}_\varepsilon(t_i) }-1\bigr)\Bigr).
\end{gather}
As  $\mathcal{A}_\varepsilon(s),$ $\bar{\mathcal{A}}_\varepsilon(s),$ and $\mathcal{C}$ vanish in the limit $\delta\rightarrow 0,$ 
the right-hand side of Eq. \eqref{zidentity} can be expanded in the form
\begin{equation}
\sum\limits_{l=0}^\infty \frac{z^l}{l!} F^{(l)}(z  M_{\mu, \varepsilon}) \,\mathcal{C}^l \, \prod\limits_{i=1}^n e^{ \omega_{\mu,1,\varepsilon}(t_i) }  \, \sum\limits_{k=0}^n \sum\limits_{(p_1<\cdots< p_k)}\prod\limits_{j=1}^{k} 
\bigl(e^{ \mathcal{A}_\varepsilon(t_{p_j}) + \bar{\mathcal{A}}_\varepsilon(t_{p_j}) }-1\bigr).
\end{equation}
Recalling the definition of $\mathcal{C}$ in Eq. \eqref{Cquantity}, we observe that to compute the ${\bf E}^*$ expectation,
\emph{i.e.} the expectation with respect to the law of ${\bar \omega}_{\delta,e,\varepsilon}(s)$ as in the proof of
Theorem \ref{theoremIDdiff}, we need to compute
\begin{equation}
{\bf E}^* \Bigl[ \mathcal{C}^l \, \prod\limits_{j=1}^{k} 
\bigl(e^{ \mathcal{A}_\varepsilon(t_{p_j}) + \bar{\mathcal{A}}_\varepsilon(t_{p_j}) }-1\bigr)\Bigr].
\end{equation}
The calculation is done by means of Lemma  \ref{BLemma}. We first write
\begin{equation}
 \mathcal{C}^l  = l!\!\!\!\! \int\limits_{\{t_{n+1}<\cdots < t_{n+l}\}} 
\prod\limits_{i=1}^l e^{\omega_{\mu,1,\varepsilon}(t_{n+i})}
\bigl(e^{\mathcal{A}_\varepsilon(t_{n+i})+\bar{\mathcal{A}}_\varepsilon(t_{n+i})}-1\bigr)\,dt_{n+1}\cdots dt_{n+l}
\end{equation}
so that it remains to calculate
\begin{equation}
{\bf E}^* \Bigl[  \prod\limits_{i=1}^l 
\bigl(e^{\mathcal{A}_\varepsilon(t_{n+i})+\bar{\mathcal{A}}_\varepsilon(t_{n+i})}-1\bigr)\,
\prod\limits_{j=1}^{k} 
\bigl(e^{ \mathcal{A}_\varepsilon(t_{p_j}) + \bar{\mathcal{A}}_\varepsilon(t_{p_j}) }-1\bigr)\Bigr].
\end{equation}
We can now apply Lemma \ref{BLemma} with $n=k+l$ provided we know the smallest and largest $t$s.
As $t_1<\cdots<t_n$ and $t_{n+1}<\cdots < t_{n+l}$ by construction, this expectation equals
\begin{equation}
\prod\limits_{i=1}^l 
\bigl(e^{\mathcal{A}_\varepsilon(t_{n+i})}-1\bigr)\,
\prod\limits_{j=1}^{k} 
\bigl(e^{ \mathcal{A}_\varepsilon(t_{p_j}) }-1\bigr)
\end{equation}
plus 
\begin{equation}
\delta \,\rho_{e,\varepsilon} \Bigl(  \max\{t_{p_k}, t_{n+l}\} - \min\{t_{p_1}, t_{n+1}\}\Bigr)
 \Bigl(\sigma^2 +  \!\!\!\!\!\!\int\limits_{\mathbb{R}\setminus \{0\}} \!\! (e^u-1)^2 d\mathcal{M}(u)\Bigr)+o(\delta), 
\end{equation}
if $k+l=2,$ and 
\begin{equation}
\delta \rho_{e,\varepsilon} \Bigl(  \max\{t_{p_k}, t_{n+l}\} - \min\{t_{p_1}, t_{n+1}\}\Bigr) \int\limits_{\mathbb{R}\setminus \{0\}}  (e^u-1)^{k+l} d\mathcal{M}(u) + o(\delta),
\end{equation}
if $k+l>2.$
Having computed the ${\bf E}^*$ expectation, the rest of the argument is the same as in the proof of Theorem \ref{theoremIDdiff}.
\end{proof} 

\section{Conclusions}
\noindent We have presented a theory of intermittency differentiation for a general class of 1D infinitely divisible
multiplicative chaos measures on the interval including the canonical measures of Bacry-Muzy as a special case.
The rule of intermittency differentiation is an exact, non-local, Feynman-Kac equation that expresses the intermittency derivative of 
the expectation of a test function of the total mass on one or more subintervals of the unit interval in terms of the derivatives 
of the function and the L\'evy-Khinchine formula of the underlying distribution. 
The equation is based on the intermittency invariance of the underlying infinitely divisible field 
that we established in full generality in this paper. This invariance is a novel technical devise
that substitutes for the non-existent Markov property of the underlying field and allows one to derive a Feynman-Kac equation
for the distribution of the total mass by considering a stochastic flow in intermittency (as opposed to time in the classical framework
of diffusions). The intermittency invariance gives two ways of evaluating the limit of the flow, which results in the differentiation rule.
The first way is the backward Kolmogorov equation for L\'evy processes, the second way involves detailed analysis of certain 
infinite series expansions, combined with a key combinatorial property of the measure that we derived in the paper. Our
analysis of these expansions is exact at the level of formal power series but not mathematically rigorous as we have not examined
their convergence properties.

Our approach naturally extends to higher derivatives. We have identified a class of non-local functionals of the limit measure
generalizing its total mass that are invariant under intermittency differentiation and derived the corresponding differentiation rule
for them. This rule allows one to compute all higher order derivatives at zero intermittency in principle. These 
non-local functionals can be thought of as measure changes of the total mass. The associated Girsanov transformation
is only known in the case of GMC and left an an open question in the general infinitely divisible case.
 
We have illustrated the rule of intermittency differentiation with two examples. The first is that of positive integer
moments of the total mass of the measure. Given the known multiple integral representation of these moments,
we rigorously proved that the intermittency differentiation rule for the moments coincides with 
the formula for the intermittency derivative that follows from the multiple integral representation.   
In the second example we derived a formula for the covariance of the total mass from the
differentiation rule for two subintervals. 

The intermittency differentiation rule is the first step towards a perturbative expansion of the distribution of the total mass and, more generally,
the dependence structure of the limit measure in powers of intermittency, \emph{i.e.} the high temperature expansion. 
We have limited ourselves in this paper to the differentiation rule for the first derivative and evaluated it explicitly. 
The computation of higher order derivatives is technically more difficult 
and left to future research as is the problem of renormalizability of the full intermittency expansion. Finally,
we also want to mention the very interesting open problem of computing positive integer moments explicitly,
\emph{i.e.} computing the generalized Selberg integral corresponding to the underlying infinitely divisible distribution
and intensity measure that is defined in the paper. 




\appendix

\section{Bacry-Muzy construction with a general intensity measure}
 
\noindent In this section we will review and somewhat extend  the more technical aspects  the infinitely divisible multiplicative chaos (IDMC) construction of Bacry and Muzy \cite{BM1} and \cite{BM}. Our setup, which we first
introduced in \cite{MeIntRen} in the gaussian case, is slightly more general than that of Bacry-Muzy \cite{BM1}, \cite{BM} in that we keep their conical set construction but allow
for a general intensity measure, subject to a positivity condition. 

The starting point is an infinitely divisible (ID) independently scattered random measure
$P$ on the time-scale plane $\mathbb{H}_+=\{(t, \, l), \, \, l>0\},$
distributed uniformly with respect to some positive intensity measure $\rho.$
The existence of such random measures is established in \cite{RajRos}.
This means that $P(A)$ is ID for measurable subsets
$A\subset\mathbb{H}_+,$ $P(A)$ and $P(B)$ are independent if $A\bigcap B=\emptyset,$ 
and
\begin{equation}
{\bf{E}} \left[ e^{i q P(A)} \right]=e^{\mu\phi(q)\rho(A)},
\,\,\,q\in\mathbb{R},
\end{equation}
where $\mu>0$ is the intermittency parameter and $\phi(q)$ is the logarithm of the characteristic
function of the underlying ID distribution as given by the L\'evy-Khinchine formula in Eq. \eqref{phi}.
The meaning of the normalization $\phi(-i)=0$ is that ${\bf E} \bigl[e^{P(A)}\bigr]=1$
for all measurable subsets $A\subset\mathbb{H}_+.$ The spectral function $\mathcal{M}(u)$ is
continuous and non-decreasing on $(-\infty, 0)$ and $(0, \infty)$, and satisfies the integrability and limit
conditions $\int_{[-1,1] \setminus \{0\}} u^2
d\mathcal{M}(u)<\infty$ and $\lim\limits_{u\rightarrow\pm\infty}
\mathcal{M}(u)=0.$ We will further assume that $\mathcal{M}(u)$
decays at infinity fast enough so that all integrals with respect to it that we need
converge, which restricts the class of permissible spectral
functions. Next, following \cite{Pulses} and \cite{Schmitt}, Bacry and Muzy \cite{BM} introduce special
conical 
sets $\mathcal{A}_{\varepsilon}(u)$ in the time-scale
plane defined by
\begin{equation}
\mathcal{A}_{\varepsilon}(u) = \left\{(t,l) \,\,\Big\vert\,\,
|t-u|\leq\frac{l}{2} \,\,\text{for}\,\,\varepsilon\leq l\leq 1
\,\,\text{and}\,\,|t-u|\leq\frac{1}{2}\,\,\text{for}\,\,l\geq
1\right\}.
\end{equation}
The sets
$\mathcal{A}_{\varepsilon}(u)$ and
$\mathcal{A}_{\varepsilon}(v)$ intersect iff $|u-v|<1.$ 

Our choice of the intensity measure will be somewhat more general
than what was originally proposed by Bacry and Muzy \cite{BM}.
Let 
\begin{equation}\label{ansatz}
\rho(dt \, dl)=\frac{f(l)}{l^2}\,dt\,dl,
\end{equation}
where the function $f(l)$ is defined by
\begin{gather}
\frac{f(l)}{l^2} = -\frac{d^2}{dl^2} \log r(l), \; l\in (0, 1), \label{fr} \\
f(l) = \frac{d}{dz}\Big\vert_{z=1} \log r(z),\;  l\geq 1,\label{fr2}
\end{gather}
in terms of some function $r(t)$ that satisfies the properties
\begin{gather}
r(t) \; \text{is positive, smooth and even} \; \text{on} \;(-1,\,0) \cup (0,\,1),  \\
\lim\limits_{t\rightarrow 0^+}\;t\frac{d}{dt} \log r(t)=1 \label{derivbound},
\end{gather}
and assume that $f(l)$ as defined in Eqs. \eqref{fr} and \eqref{fr2} is positive.
For example, the canonical choice of Bacry and Muzy corresponds to
\begin{align}
r(t)=&|t|, \\
f(l)=&1.
\end{align}
Let the function $\rho_{\varepsilon}(u, v)$ denote the intensity measure of intersections
of the conical sets
\begin{equation}
\rho_{\varepsilon}(u, v) =
\rho\left(\mathcal{A}_{\varepsilon}(u)\bigcap\mathcal{A}_{\varepsilon}(v)\right).
\end{equation}
Clearly, $\rho_{\varepsilon}(u, v)$ is an even function of $u-v$
so that we can write $\rho_{\varepsilon}(u,
v)=\rho_{\varepsilon}(|u-v|).$ It is easy to show, cf. the Appendix of \cite{MeIntRen}, that it is given by
\begin{equation}\label{rhogeneral}
\rho_{\varepsilon}(u) =
\begin{cases}
-\log r(u)& \, \text{if $\varepsilon\leq |u|\leq 1$}, \\
- \log r(\varepsilon) + \bigl(1-\frac{|u|}{\varepsilon}\bigr) \varepsilon\frac{d}{d\varepsilon} \log r(\varepsilon) & \, \text{if
$|u|<\varepsilon$},
\end{cases}
\end{equation}
and it is identically zero for $|u|>1.$
It is clear that
$u\rightarrow P\left(\mathcal{A}_{\varepsilon}(u)\right)$
is a stationary, ID process such that 
$P\left(\mathcal{A}_{\varepsilon}(u)\right)$ and $P\left(\mathcal{A}_{\varepsilon}(v)\right)$ are dependent iff $|u-v|<1.$ 
One can show that with probability one, the process
$u\rightarrow P\left(\mathcal{A}_{\varepsilon}(u)\right)$ has right-continuous
trajectories with finite left limits.

Given these preliminaries, the IDMC measure $M_{\mu}(dt)$ on the interval $[0,\,1]$ 
associated with $\phi(q)$ at intermittency $\mu$ is the zero regularization scale limit $\varepsilon\rightarrow 0$ of
finite scale random measures that are
defined to be the exponential functional of the $u\rightarrow P\left(\mathcal{A}_{\varepsilon}(u)\right)$ process.
To simplify notations, let
\begin{equation}\label{omegaprocess}
\omega_{\mu,\varepsilon}(u) \triangleq P\bigl(\mathcal{A}_{\varepsilon}(u)\bigr).
\end{equation}
Then, the theorem of Bacry and Muzy states that the limit
\begin{equation}
M_{\mu}(a, b)=\lim\limits_{\varepsilon\rightarrow 0}\int\limits_a^b \exp\bigl(\omega_{\mu,\varepsilon}(u)\bigr) \, du,
\end{equation}
exists in the weak a.s. sense. It was formally established in
\cite{BM1} based on the theory of Kahane \cite{K2} using the normalization of $\phi(q)$ and the property of $P$ of
being independently scattered, which guarantee the martingale property of the construction,
\begin{equation}
{\bf E}\left[M_{\mu, 
\varepsilon'}(a, b)\,\,\vert\,\,\mathcal{F}_\varepsilon\right] =
M_{\mu, \varepsilon}(a, b), \,\,\varepsilon'<\varepsilon,
\end{equation}
where $\mathcal{F}_\varepsilon$ is the sigma algebra generated by
$P(dt\,dl), \,\,l>\varepsilon.$
The limit measure is then
non-degenerate in the sense of ${\bf E}[M_{\mu}(a, b)]=|b-a|$
under the assumption of Eq. \eqref{nondeg}.

The best known analytical handle on the Bacry-Muzy construction is given in the following fundamental lemma due to \cite{BM1},
which we state in the slightly greater generality than the original as we allow for $f(l)$ as in Eqs. \eqref{fr} and \eqref{fr2}.
\begin{lemma}[Main lemma]\label{main}
Given $t_1\leq\cdots\leq t_n$ and $q_1,\cdots, q_n,$ the joint
characteristic function of $\omega_{\mu,\varepsilon}(t_j),$
$j=1\cdots n,$ is
\begin{equation}\label{mainchar}
{\bf E}\Biggl[\exp\Bigl(i\sum_{j=1}^n q_j \omega_{\mu,\varepsilon}(t_j)\Bigr)\Biggr] = \exp\Bigl(\mu\sum_{p=1}^n
\sum_{k=1}^p \alpha_{p,k} \,\rho_{\varepsilon}(t_p-t_k)\Bigr),
\end{equation}
where $\rho_\varepsilon(u)$ is defined in Eq. \eqref{rhogeneral} and
the coefficients $\alpha_{p,k}$ are given in terms of $\phi(q)$ by
\begin{equation}\label{alpha}
\alpha_{p,k} = \phi(r_{k,p}) + \phi(r_{k+1,p-1}) - \phi(r_{k,p-1}) - \phi(r_{k+1,p}),\; 
\end{equation}
and $r_{k, p} = \sum_{m=k}^p q_m$ if $k\leq p$ and zero otherwise. In addition,
\begin{equation}\label{sumID}
\sum_{p=1}^n \sum_{k=1}^p \alpha_{p,k} = \phi\Bigl(\sum_{j=1}^n
q_j\Bigr).
\end{equation}
\end{lemma}
The significance of this lemma cannot be overemphasized as it
implies the self-similarity
property of the limit measure, determines its positive integer moments, and
is the source of all known invariances of the $\omega_{\mu,
\varepsilon}(t)$ process. 
 
A multiple integral representation of the positive integer moments of the total mass of the limit measure can be written down explicitly for the general $r(t).$ Given $d(m)$ as in Eq. \eqref{d}, 
the $n$th moment of the total mass is given by a generalized Selberg integral of dimension $n.$
Let $0\leq a<b\leq 1$ and $n$ satisfy Eq. \eqref{finmom}. 
\begin{equation}\label{rsinglemomformula}
{\bf E}\Bigl[\Bigl(\int\limits_a^b M_\mu(dt)\Bigr)^n\Bigr] = n!
\int\limits_{\{a<t_1<\cdots<t_n<b\}} \prod\limits_{k<p}^n r(t_p-t_k)^{-\mu\,d(p-k)}
\,dt^{(n)}.
\end{equation}
The argument is based on a simple application of Lemma \ref{main} in the form of the identity
\begin{equation}\label{momrhoidentity}
\exp\Bigl(\sum\limits_{k<p}^n \mu\,d(p-k) \rho_\varepsilon(t_p-t_k)\Bigr) = {\bf E}\Bigl[
e^{ \omega_{\mu,\varepsilon}(t_1)+\cdots+\omega_{\mu,\varepsilon}(t_n)}\Bigr] 
\end{equation}
for any $0<t_1<\cdots<t_n<1,$
and Fubini's theorem. The coefficients $d(m)$ have the important
property
\begin{equation}
\sum\limits_{k<p}^n d(p-k) = \phi(-in).
\end{equation}
This equation combined with $r(t)\thicksim t$ in the limit $t\rightarrow 0,$ cf. Eq. \eqref{derivbound}, implies
\begin{equation}\label{momentscaling}
{\bf E}\Bigl[\Bigl(\int_0^t M_\mu(ds)\Bigr)^n\Bigr] \thicksim const\, t^{n-\mu\phi(-in)}, \; t\rightarrow 0,
\end{equation}
which we now obtained without relying on the self-similarity property of the measure, cf. Eq. \eqref{selfsim} below.
 
We note parenthetically that the formula in Eq. \eqref{rsinglemomformula} suggests how to define the Selberg
integral in the general ID case.
\begin{align}
S_n\bigl(\lambda, \lambda_1,\lambda_2\bigr) \triangleq &\int\limits_{0<t_1<\cdots<t_n<1} \prod\limits_{i=1}^n  r(t_i)^{\lambda_1 d(i)}r(1-t_i)^{\lambda_2 d(n-i+1)}\times \nonumber \\ &\times \prod\limits_{k<p}^n r(t_p-t_k)^{2\lambda\,d(p-k)}
\,dt^{(n)} \label{Sdef}
\end{align}
for generally complex $\lambda,$ $\lambda_1,$ and $\lambda_2.$ Some of its properties in the general ID case are derived in \cite{MeLMP}.

We end this review with a brief comment about the origin of the self-similarity property of the measure, 
which is the primary motivation for the IDMC construction as we explained in the Introduction.
This result is specific to the Bacry-Muzy intensity measure in Eq. \eqref{rho}. 
Introduce a L$\acute{\text{e}}$vy process (a stochastic process with
stationary, independent increments) $\delta\rightarrow X(\delta)$
that is independent of the $t\rightarrow \omega_{\mu, 
\varepsilon}(t)$ process and defined in terms of the ID distribution
associated with $\phi(q)$ as follows
\begin{equation}\label{Xprocessdef}
{\bf E}\left[e^{i q X(\delta)}\right] = e^{\delta\phi(q)}, \,\,
X(0)=0.
\end{equation}
The existence and uniqueness of $X(\delta)$ follow from the general
theory of ID processes, confer \cite{Bertoin}. 
Then, as a corollary of Lemma \ref{main}, there holds the
following invariance of the field $\omega_{\mu, \varepsilon}(t)$ with respect to the regularization scale parameter,
which is understood to be the equality in law of stochastic processes in $t$ on the interval $t\in[0, 1].$
\begin{equation}
X(\delta)+\omega_{\mu, \varepsilon}(t) = \omega_{\mu, \varepsilon e^{-\delta/\mu}}(t e^{-\delta/\mu}).
\end{equation}
An elementary change of variables argument given in \cite{BM}
shows that this invariance implies 
stochastic self-similarity of the limit process,
\begin{equation}\label{selfsim}
M_{\mu}(0, t)= t\exp\bigl(X(-\mu\log t)\bigr) M_{\mu}(0, 1),
\end{equation}
understood as the equality of random variables in law at fixed
$t<1.$ It must be emphasized that self-similarity alone
does not capture the
law of the total mass $M_{\mu}(0,1)$ but only of $M_{\mu}(0, t)$ in terms of
$M_{\mu}(0, 1).$  
Hence, the multiplier in Eq. \eqref{cascade} is given by 
\begin{equation}
W_\gamma = \gamma e^{X(-\mu\log \gamma)},
\end{equation}
and is log-infinitely divisible as expected.



\end{document}